\documentclass[reqno, 12pt, reqno]{amsart}

\usepackage{amsfonts, amsthm, amsmath, amssymb}

\usepackage{tikz}
\tikzset{node distance=2cm, auto}

\usepackage{hyperref}
\usepackage{a4wide}

\RequirePackage{mathrsfs} \let\mathcal\mathscr

\numberwithin{equation}{section}

\newtheorem{theorem}{Theorem}[section]
\newtheorem*{theorem*}{Theorem}
\newtheorem{lemma}[theorem]{Lemma}
\newtheorem{proposition}[theorem]{Proposition}

\theoremstyle{definition}
\newtheorem*{ack}{Acknowledgements}

\newtheorem*{rem*}{Remark}
\newtheorem*{grem}{General remarks}
\newtheorem{definition}[theorem]{Definition}

\renewcommand{\rho}{\varrho}
\newcommand{\1}{\mathbf{1}}
\newcommand{\Br}{{\rm Br}}

\newcommand{\0}{\mathbf{0}}

\newcommand{\PP}{\mathbb{P}}
\renewcommand{\AA}{\mathbb{A}}

\newcommand{\ZZ}{\mathbb{Z}}

\newcommand{\NN}{\mathbb{N}}

\newcommand{\QQ}{\mathbb{Q}}
\newcommand{\RR}{\mathbb{R}}
\newcommand{\CC}{\mathbb{C}}

\newcommand{\cA}{\mathcal{A}}

\newcommand{\beps}{\boldsymbol{\epsilon}}

\renewcommand{\leq}{\leqslant}

\renewcommand{\geq}{\geqslant}

\renewcommand{\a}{\mathbf{a}}
\renewcommand{\b}{\mathbf{b}}

\newcommand{\m}{\mathbf{m}}
\newcommand{\bn}{\mathbf{n}}

\renewcommand{\u}{\mathbf{u}}

\newcommand{\x}{\mathbf{x}}
\newcommand{\y}{\mathbf{y}}

\DeclareMathOperator{\sign}{sign}

\DeclareMathOperator{\nf}{\mathbf{N}}

\newcommand{\Mod}[1]{\;(\operatorname{mod}\,#1)}

\newcommand{\id}{\rm{id}}

\newcommand\bN{\mathbf{N}}

\DeclareMathOperator{\vol}{vol}

\newcommand{\eps}{\varepsilon}

\begin{document}

\title[The square-free representation function of a norm form]
{On the square-free representation function \\ of a norm form and 
nilsequences}
\author{Lilian Matthiesen}
\address{Institut de Math\'ematiques de Jussieu --- Paris Rive Gauche\\
UMR 7586\\
B\^atiment Sophie Germain\\
Case 7012\\
75205 Paris Cedex 13\\
France}
\email{lilian.matthiesen@imj-prg.fr}

\thanks{While working on this paper the author was supported by a postdoctoral 
fellowship of the Fondation Sciences Math{\'e}matiques de Paris}

\begin{abstract}
We show that the restriction to square-free numbers of the representation 
function attached to a norm form does not correlate with nilsequences. 
By combining this result with previous work of Browning and the author, we obtain 
an application that is used in recent work of Harpaz and Wittenberg on 
the fibration method for rational points.
\end{abstract}

\maketitle

\section{Introduction}
Let $K_1, \dots, K_r$ be finite extensions of $\QQ$ of degree at least $2$ and 
let $n_1, \dots, n_r \geq 2$ denote their respective degrees.
For each $1 \leq i \leq r$ let $\{\omega_1^{(i)},\dots,\omega_{n_i}^{(i)}\}$ be a 
$\ZZ$-basis for the ring of integers of $K_i$ and 
denote by
$$
\nf_{K_i}(x_1,\dots,x_{n_i})=
N_{K_i/\QQ}(x_1\omega_1+\dots+x_{n_i}\omega_{n_i})
$$ 
the corresponding norm form, where $N_{K_i/\QQ}$ is the field norm.
One of the central results from \cite{bm}, stated as \cite[Theorem 1.3]{bm}, 
proves weak approximation for varieties
$X \subset \AA_{\QQ}^{n_1 + \dots + n_r + s}$ defined by the 
system of equations
\begin{equation}\label{eq:system}
0 \not= 
\nf_{K_i}(x_1^{(i)},\dots,x_{n_i}^{(i)})
= f_i(u_1, \dots, u_s), \qquad (1 \leq i \leq r),
\end{equation}
where $s\geq 2$ and $f_1, \dots, f_r \in \ZZ[u_1,\dots,u_s]$ 
are pairwise non-proportional linear forms.
This weak approximation result is deduced from an asymptotic formula for the 
number of (suitably restricted) integral points on $X$ 
(see \cite[Theorem 5.2]{bm}).

Our aim here is to develop refinements of both \cite[Theorem 5.2]{bm} and the 
weak approximation result, that allow one to deduce the linear case of a 
conjecture due to Harpaz and Wittenberg \cite{HW} 
(see \cite[\S9]{HW} for details).
Building on this linear case, they establish the following very strong fibration 
theorem for the existence of rational points.

\begin{theorem*}[Harpaz--Wittenberg \cite{HW}, Theorem 9.27]
Let $X$ be a smooth, proper, irreducible variety over $\QQ$ and let
$f : X \to \PP^1_{\QQ}$ be a dominant morphism, with rationally connected 
geometric generic fiber.
Suppose further that all non-split fibers lie over rational points of 
$\PP^1_{\QQ}$. 
If $X_c(\QQ)$ is dense in $X_c(\AA_{\QQ})^{\Br(X_c)}$ for
every rational point $c$ of a Hilbert subset of $\PP^1_{\QQ}$,
then $X(\QQ)$ is dense in $X(\AA_{\QQ})^{\Br(X)}$.
\end{theorem*}

The results we discuss here analyse instead of \eqref{eq:system} 
the following system of equations:
\begin{equation}\label{eq:system-sqf}
0 \not= 
\nf_{K_i}(x_1^{(i)},\dots,x_{n_i}^{(i)})
= f_i(u_1, \dots, u_s) 
\mu(f_i(u_1, \dots, u_s))^2,
\qquad (1 \leq i \leq r),
\end{equation}
where $\mu$ denotes the M\"obius function.
Note that counting integral solutions to this system is a question concerning the 
representation of square-free integers by norm forms.

The weak approximation type result relevant to \cite{HW} is the following.
\begin{theorem} \label{t:WA}
Let $K_1, \dots, K_r$ be finite extensions of $\QQ$ of degree at least $2$.
Let $f_1, \dots,f_r \in \ZZ[u_1, \dots,u_s]$ be pairwise non-proportional linear 
forms.  
Let $S$ be a finite set of primes that contains all primes 
$p \leq C$ for some constant $C$ only depending on
$f_1, \dots,f_r$ and $K_1, \dots, K_r$.
Let $\u \in \ZZ^s$ be a vector such that $f_i(\u)$ is non-zero and a local 
integral norm from $K_i$ at all places of $S$ and also at the real place.
Then there exists a vector $\u' \in \ZZ^s$ such that
\begin{enumerate}
\item $\u'$ is arbitrarily close to $\u$ at the places of $S$;
\item $\u'$ belongs to any given open convex cone of $\RR^s$ which contains $\u$;
\item $f_i(\u')$ is `square-free outside of $S$' in the sense that 
$v_p(f_i(\u)) \geq 2$ only if $p \in S$,
and $f_i(\u')$ is the norm of an integral element of $K_i$, for all 
$i$.
\end{enumerate}
\end{theorem}

Just as in the case of the weak approximation result from \cite{bm}, 
this result is a corollary to an asymptotic formula for the number of (suitably 
restricted) integral solutions to \eqref{eq:system-sqf} which we state as Theorem 
\ref{t:NB} below.  
The deduction of Theorem \ref{t:WA} from Theorem \ref{t:NB} will be carried out 
in Section \ref{s:local}.
In order to state Theorem \ref{t:NB}, we proceed by introducing a square-free 
representation function for any given norm form $\nf_{K}$ associated to a 
field extension $K/\QQ$ of degree $n\geq2$.

Let $\{\omega_1, \dots, \omega_n\}$ denote the basis with respect to which 
$\nf_{K}$ 
is defined. 
As in \cite[\S 2]{bm}, we let $\mathfrak{D}_{+} \subset \RR^n$ denote a 
fundamental domain for the equivalence relation that identifies 
two vectors $\mathbf x$ and $\mathbf y$ if and only if 
$x_1\omega_1+\dots+x_n\omega_n$ and 
$y_1\omega_1+\dots+y_n\omega_n$ are associated by a unit of positive 
norm.
Define the representation function $R:\ZZ \to \ZZ_{\geq 0}$ by setting
\begin{equation} \label{eq:def-R}
R(m)=
\1_{m \not= 0} \cdot~
\#\left\{\mathbf{x} \in \ZZ^{n}\cap\mathfrak{D}_{+}:
\begin{array}{l}
\bN_{K}(\x)=m
\end{array}
\right\},
\end{equation}
for any $m \in \ZZ$.
This is a special case of the representation functions considered in
\cite[Definition 5.1]{bm}.
Here, we will be interested in the following restrictions of $R$.
\begin{definition}
Let $R$ be the function defined in \eqref{eq:def-R}, and let $S$ be 
a finite set of primes.
Then we let $R^*_{S}:\ZZ \to \ZZ_{\geq 0}$ denote the restriction of $R$ 
to integers $m$ that are square-free outside $S$.
That is, we define 
\begin{align*}
 R^{*}_{S}(m) = 
 \mu^2\bigg(\prod_{p \not\in S } p^{v_p(m)}\bigg) R(m)
\end{align*}
for $m \in \ZZ$.
\end{definition}
\begin{rem*}
The representation function $R$ is in general not multiplicative, not even away 
from $S$.
For this reason, inclusion-exclusion arguments cannot be carried out in a 
straightforward way, and therefore the main results of this paper do not follow 
directly from those in \cite{bm}.
\end{rem*}

Finally, we define, as in \cite[\S4]{bm}, a local count of representations 
by setting  
$$
\rho(q,A)= 
\left\{
\mathbf{x} \in (\ZZ/q\ZZ)^{n}:
\begin{array}{l}
\bN_{K}(\x)\equiv A \Mod{q} 
\end{array}
\right\}
$$
for any $q\in \NN$ and $A \in \ZZ/q\ZZ$.

With this notation, the following result is the asymptotic result for $R^*_S$ 
that corresponds to \cite[Theorem 5.2]{bm}.
\begin{theorem}\label{t:NB}
Let $K_1, \dots, K_r$ be finite extensions of $\QQ$ of degree at least $2$ and 
let $n_1, \dots, n_r$ denote their respective degrees.
Let $S_1, \dots, S_r$ be finite sets of primes. 
For each $i \in \{1, \dots, r\}$, let $R_i$ be the representation function of a 
norm form associated to $K_i/\QQ$, and let $R^*_i := {R^*_{i}}_{S_i}$ denote its 
restriction to integers that are square-free outside $S_i$. 
Let $\mathfrak{K} \subset [-1,1]^s$ be a convex body.
Further, suppose that $f_1, \dots, f_r \in \ZZ[u_1,\dots,u_s]$ are 
pairwise non-proportional linear forms and assume that 
$|f_i(\mathfrak{K})|\leq 1$, for $1\leq i\leq r$.
Given any modulus $q \in \NN$ and a vector $\a \in (\ZZ/q\ZZ)^s$ such that
$v_p(f_i(\a)) < v_p(q)$ for all $p \in S_i$ and 
$i \in \{1, \dots, r\}$, we then have
$$
\sum_{\substack{\u \in \ZZ^s \cap T \mathfrak{K}\\ \u \equiv \a \Mod{q}}}
\prod_{i=1}^r R^*_i(f_i(\u))
= \beta_\infty \prod_p \beta_p \cdot T^s+o(T^{s}), 
\quad (T\rightarrow \infty),
$$
where
$$ 
\beta_{\infty}
=\sum_{\beps \in \{\pm\}^{r}}
 \vol\left(
  \mathfrak{K} \cap 
  \mathbf{f}^{-1}(\RR_{\epsilon_1} \times\dots\times \RR_{\epsilon_r})
 \right)
\prod_{i=1}^r \kappa_i^{\epsilon_i}
$$
with 
$$
\kappa_i^{\epsilon_i}
= \vol
 \{\x \in \mathfrak{D}^+_i : 0< \epsilon_i \bN_{K_i}(\x) \leq 1 \}
\quad \text{and} \quad  \mathbf{f} = (f_1, \dots, f_r): \ZZ^s \to \ZZ^r,
$$
and 
$$
\beta_p=
 \lim_{m\rightarrow \infty} 
 \frac{1}{p^{ms}}\sum_{
 \substack{
 \u\in (\ZZ/p^m\ZZ)^s\\
 \u\equiv\a\Mod{p^{v_p(q)}}}}
 \prod_{i=1}^r
 \bigg( 1 - \1_{p \not\in S_i} \frac{\rho_i(p^2,0)}{p^{2n_i}} \bigg)
 \frac{\rho_i(p^m,f_i(\u))}{p^{m(n_i-1)}},
$$
for each prime $p$.
Furthermore, the product $\prod_p \beta_p$ is absolutely convergent.
\end{theorem}

\begin{rem*}
In all of this work, $R^*_S$ could be replaced by the more general 
representation function that arises from replacing $R$ by a function 
$R(\mathfrak{X},M,\b)$ as defined in \cite[Definition 5.1]{bm}.
This would allow one to prove a weak approximation result that does not only 
take the variables $\u$ from \eqref{eq:system-sqf} into account, but also the 
variables $\x_i$.
While working with a general function $R(\mathfrak{X},M,\b)$ requires essentially 
no additional work, we decided to restrict ourselves here to the special case of 
\eqref{eq:def-R} for reasons of notational simplicity.
\end{rem*}

The proof of \cite[Theorem 5.2]{bm} uses the methods that were 
introduced in Green--Tao \cite{GT-linearprimes}, and as such the two main steps 
in the proof are the construction of a family of pseudo-random majorants for a 
$W$-tricked version of the function $R$ and the proof that this new function is 
orthogonal to nilsequences.
The pseudo-random majorants constructed in \cite{bm} for the $R$-function are in 
fact pseudo-random majorants for the $W$-tricked version of $R^*_S$ as well. 
Thus, the main step that is missing is to check orthogonality with nilsequences.

We prove the convergence of the product of local factors in Section \ref{s:local}.
As all remaining parts work exactly as in \cite{bm}, our main focus here is to 
establish a non-correlation estimate (Theorem \ref{p:nilsequences} 
below) for $R^*_S$ that corresponds to \cite[Proposition 6.3]{bm} with $R$ 
replaced 
by its square-free version.
Both, Section \ref{s:progressions} and Section \ref{s:technical} contain 
technical lemmas needed in the proof of Theorem \ref{p:nilsequences} in 
Section \ref{s:proof}.

The statement of Theorem \ref{p:nilsequences} requires a `$W$-trick'.
In contrast to the case of $R$ handled in \cite{bm}, there is a lot of 
flexibility in the choice of the $W$-trick here, as the exceptionally large values 
of $R$, which were problematic before, occur at integers that are not 
square-free. 
Given any integer $N>0$, let $w(N)= \log \log N$ and set 
\begin{equation} \label{def:W}
W(N)
= \prod_{p \leq w(N)} p^{\alpha(p)}, 
\end{equation}
where $\alpha(p) \in \NN$ is such that
$$
p^{\alpha(p)-1} < \log N \leq p^{\alpha(p)}.
$$
Observe that $\alpha(p)>1$.
Taking $N$ sufficiently large allows us to assume that any given finite set $S$ 
of primes is contained in the set of primes less than $w(N)$ and 
moreover that $q| W(N)$ for any given integer $q$. 
Given a representation function $R^*_S$ and any integer $T>0$, we define the 
following set of `unexceptional' residues:
\begin{equation}\label{eq:def-A}
\cA(R^*_S, N)=\left\{A \bmod{W(N)}: 
\begin{array}{ll}
 0 \leq v_p(A) \leq 1 &\text{if }p<w(N), p\not\in S \cr
 0 \leq v_p(A) < v_p(W(N))/3 &\text{if } p\in S \cr
 \rho(W(N),A) >0
\end{array}
\right\}. 
\end{equation}
With the exception of integers that are divisible to a large order by some 
prime $p$ from $S$, the support of $R^*_S$ is contained in the set of numbers 
whose residues modulo $W(N)$ belong to $\cA(R^*_S, N)$.
Since $S$ is finite, we can avoid the exceptional set by fixing 
the $S$-part $\prod_{p\in S} m^{v_p(m)}$ of integers $m$ under consideration and 
taking $N$ sufficiently large so that $v_p(m) < v_p(W(N))/3$.

With the $W$-trick in place, we are now ready to reveal the main result of this 
paper, which states that the $W$-tricked version of $R^*_S$ is orthogonal to 
nilsequences.
\begin{theorem}\label{p:nilsequences}
 Let $G/\Gamma$ be a nilmanifold of dimension $m_G \geq 1$, let
$G_{\bullet}$ be a filtration of $G$ of degree $\ell \geq 1$, and let
$g \in \mathrm{poly}(\ZZ,G_{\bullet})$ be a polynomial sequence.
Suppose that $G/\Gamma$ has a $Q$-rational Mal'cev basis $\mathcal X$
for some $Q \geq 2$, defining a metric $d_{\mathcal X}$ on $G/\Gamma$.
Suppose that $F: G/\Gamma \to [-1,1]$ is a Lipschitz function.
Let $N$ and $T=T(N)$ be positive integer parameters that satisfy 
$N^{1 - \eps} \ll_{\eps} T \leq N$ for all $\eps > 0$.
Then for $\epsilon\in \{\pm\}$, 
$W=W(N)$, and $A \in \ZZ$ with 
$A \Mod{W} \in \cA(R^*_S, N)$ and $0 \leq \epsilon A < W$,
we have the estimate
\begin{align*}
\bigg| \frac{W}{T}
  \sum_{0< \epsilon m \leq T/W} 
  \bigg(&
    R^*_S(W m +A) -  \kappa^{\epsilon}
    \frac{\rho(W,A)}{W^{n-1}}
    \prod_{p>w(N)} 
    \bigg( 1 -
    \frac{\rho(p^2,0)}{p^{2n}} \bigg)
  \bigg)
  F(g(|m|)\Gamma)
\bigg| \\
&\ll_{m_G,\ell,E}
  Q^{O_{m_G,\ell,E}(1)}
 \frac{1+ \|F\|_{\mathrm{Lip}}}{(\log \log \log N)^{E}}
 \frac{\rho(W,A)}{W^{n-1}}
 \prod_{p>w(N)} \bigg( 1 - \frac{\rho(p^2,0)}{p^{2n}} \bigg),
\end{align*}
for any $E>0$ and provided $N$ is sufficiently large.
\end{theorem}

\begin{grem}
We assume familiarity with \cite{bm} throughout this paper.
The ideas and proofs that we present in Sections \ref{s:local} and \ref{s:proof} 
are very closely related to the material from \cite{bm}.
The main new observation is the fact that these ideas can be made to work 
in the case of the square-free representation function by means of the new 
technical lemmas we prove in Sections \ref{s:progressions} and \ref{s:technical}.
\end{grem}

\begin{ack}
 This paper owes its existence to questions from Yonatan Harpaz and Olivier 
Wittenberg. I am grateful to both of them and to Tim Browning for their comments 
on an earlier version of this paper, to Alexei Skorobogatov for initiating 
the discussions and to the referee for valuable suggestions and comments that 
helped improve the paper.
\end{ack}

\section{Local factors and the deduction of Theorem \ref{t:WA}} \label{s:local}

The aim of this section is to deduce Theorem \ref{t:WA} from Theorem \ref{t:NB}.
This deduction partially relies on the following proposition, which 
asymptotically evaluates the local factors from Theorem \ref{t:NB}.
Note that this proposition implies the final part of Theorem \ref{t:NB}, namely 
that the product of local factors is absolutely convergent. 

\begin{proposition} \label{p:beta}
Let 
$L = \max_{1 \leq i \leq r}\{\|f_i\|,s,r,n_i,|D_{K_i}|\},$ where 
$\|f_i\|$ denotes the maximum of the absolute values of the coefficients of $f_i$.
Then the local factors from Theorem \ref{t:NB} satisfy 
\begin{enumerate}
 \item $\beta_p= 1 + O_L(p^{-2})$ whenever $p \nmid q$, and
 \item $\beta_p= O_{L,q}(1)$ at all primes $p$.
\end{enumerate}
In particular, there is $L'=O_L(1)$ such that $\beta_p > 0$ provided $p \nmid q$ 
and $p \geq L'$.
\end{proposition} 
\begin{proof}
For every prime $p$ let
$$
\beta'_p =
 \lim_{m\rightarrow \infty} 
 \frac{1}{p^{ms}}\sum_{
 \substack{
 \u\in (\ZZ/p^m\ZZ)^s\\
 \u\equiv\a\Mod{p^{v_p(q)}}}}
 \prod_{i=1}^r
 \frac{\rho_i(p^m,f_i(\u))}{p^{m(n_i-1)}}.
$$
Then
$$
\beta_p
=\beta'_p
 \prod_{i=1}^r
 \bigg( 1 - \1_{p \not\in S_i} \frac{\rho_i(p^2,0)}{p^{2n_i}} \bigg)
$$
and $\beta'_p$ is the local factor that appears in \cite[Theorem 5.2]{bm}, but 
with $\rho_i(p^m,f_i(\u);p^{v_p(q)})$ replaced by $\rho_i(p^m,f_i(\u))$.
The proof of \cite[Proposition 5.5]{bm} implies that
$\beta'_p = O_{L,q}(1)$ and that 
$\beta'_p= 1 + O_L(p^{-2})$ whenever $p \nmid q$.
Indeed, the second part of \cite[Proposition 5.5]{bm} rests on 
the bound $\rho_i(p^m,f_i(\u);p^{v_p(q)}) \leq \rho_i(p^m,f_i(\u))$ and 
therefore includes a proof of the assertion (1);
assertion (2) follows from a direct application of \cite[Proposition 5.5]{bm} 
with $M=q$ since 
$\rho_i(p^m,f_i(\u);p^{v_p(q)})= \rho_i(p^m,f_i(\u))$ when $p \nmid q$.
Thus, it remains to show that
$$
 1 - \frac{\rho_i(p^2,0)}{p^{2n_i}}
= 1 + O(p^{-2})
$$
for all primes $p$ and $i \in \{1, \dots, r\}$.
By \cite[Lemma 4.5]{bm}, there is $C>0$ such that
$$
\frac{\rho_i(p^2,0)}{p^{2(n_i-1)}} \leq C 2^{n_i},
$$
that is 
\begin{equation}\label{eq:rho-p^2-0}
\frac{\rho_i(p^2,0)}{p^{2n_i}} \leq C \frac{2^{n_i}}{p^2}, 
\end{equation}
for each $1 \leq i \leq r$.
\end{proof}

We are now in the position to prove Theorem \ref{t:WA}.
The proof below is similar to that given in \cite[\S5.3]{bm}, but significantly 
easier, since we only consider weak approximation in the variables $\u$ and not 
in the variables $\x_i$.  

\begin{proof}[Proof of Theorem \ref{t:WA} assuming Theorem \ref{t:NB}]
 First of all, we may assume that $S$ contains all primes $p \leq L'$, where $L'$ 
is given by Proposition \ref{p:beta} above.
For $1 \leq i \leq r$ let $R_i$ be the representation function of some norm form 
$\nf_{K_i}$ and let $R^*_i$ be its restriction to integers $m$ that are 
square-free outside $S$.
Then it suffices to show that, given any $\eps> 0$, there exists a vector 
$\u' \in \ZZ^s$ such that
\begin{enumerate}
 \item $\quad |\u' - \u|_p < \eps \quad$ for all  $p\in S$,
 \item $\quad |t\u' - \u|  < \eps \quad$ for some $t>0$, and
 \item $\quad \prod_{i=1}^r R^*_i(f_i(\u'))>0.$
\end{enumerate}
For every $\eps>0$ there is an integer $Q$ composed of primes from $S$ such that
condition (1) is implied by the congruence
$$
\u' \equiv \u \Mod{Q} 
$$
and such that $v_p(Q) > v_p(f_i(\u))$ for every $p \in S$.
Further, let 
$$\mathfrak{K}(\u; \eps) = \{\mathbf{v} \in \RR^s: |\u - \mathbf{v}|< \eps\}$$ 
and note that whenever $T>0$ and $\u' \in T\mathfrak{K}(\u; \eps)$, then the 
second condition is satisfied.
Thus it is enough to show that
$$
N(T)=
\sum_{\substack{\u' \in \ZZ^s \cap T \mathfrak{K}(\u; \eps')
      \\ \u' \equiv \u \Mod{Q} }}
\prod_{i=1}^r R^*_i(f_i(\u'))>0
$$
for some value of $T>0$.
Theorem \ref{t:NB} implies that 
$$
N(T) \geq 
T^s 
\vol(\mathfrak{K}(\u,\eps) \cap 
     \mathbf{f}^{-1}(\RR_{\epsilon_1}\times \dots \times \RR_{\epsilon_r})) 
\kappa_1^{\epsilon_1} \dots \kappa_r^{\epsilon_r}
\prod_p \beta_p 
+ o(T^s), 
$$
where $\epsilon_i= \sign{f_i(\u)}$.
Hence, the result follows by taking $T$ sufficiently large, provided we can show 
that the product of local factors on the right hand side is positive.

To see this, first note that every factor 
$\kappa_i^{\epsilon_i}
=\vol \{\x \in \mathfrak{D}^+_i : 0< \epsilon_i \bN_{K_i}(\x) \leq 1 \}
$ is positive since $f_i(\u) \not= 0$ is a local norm from $K_i$ at 
the real place. 
Indeed, there is a vector $\x_i \in \RR^{n_i} \cap \mathfrak{D}^+_i$ such that 
$f_i(\u)=\bN_{K_i}(\x_i)$
and since $\epsilon_i= \sign{f_i(\u)}$, we have
$$t\x_i \in \{\x \in \mathfrak{D}^+_i : 0< \epsilon_i \bN_{K_i}(\x) < 1 \}$$ 
for every sufficiently small $t>0$. 
By continuity of $\bN_{K_i}$, the above set is open and has positive volume since 
it is non-empty.

Since $f_i(\u) \not= 0$ for all $i \in \{1, \dots, r\}$, the open set 
$$
\mathfrak{K}(\u,\eps) \cap 
     \mathbf{f}^{-1}(\RR_{\epsilon_1}\times \dots \times \RR_{\epsilon_r}))
$$
contains $\u$, which again implies
$\vol(\mathfrak{K}(\u,\eps) \cap 
     \mathbf{f}^{-1}(\RR_{\epsilon_1}\times \dots \times \RR_{\epsilon_r}))>0$.

By Proposition \ref{p:beta}, we have 
$\prod_{p \not\in S} \beta_p > 0$, so that it remains to check that
$\beta_p>0$ whenever $p \in S$.
We proceed as in \cite[\S5.3]{bm}:

Let $p$ be any element from $S$ and recall that for $1 \leq i \leq r$ there 
is an integral element $k_i \in K_i$ such that 
$f_i(\u)=N_{K_i\otimes_{\QQ}\QQ_p/\QQ_p}(k_i)$.
This implies that for every $m>0$ there is a vector $\x_i \in \ZZ^{n_i}$ 
such that 
$$ f_i(\u) \equiv \nf_{K_i}(\x_i) \Mod{p^m}$$
and hence
$$
\prod_{i=1}^r \rho_i(p^{m}, f_i(\u))\geq 1.
$$
Choosing
$$
m=
2\bigg(1 +v_p(Q)
     + \sum_{i=1}^r v_p(f_i(\u)) 
     + \sum_{i=1}^rv_p(n_i)
 \bigg),
$$
we apply \cite[Lemma 3.4]{bm} with $A=f_i(\u)$, $G=\nf_{K_i}$ and $\ell= 0$ to 
deduce that
$$
\prod_{i=1}^r \frac{\rho_i(p^{m'}, f_i(\tilde \u))}{p^{m'(n_i-1)}}
= \prod_{i=1}^r \frac{\rho_i(p^{m}, f_i(\u))}{p^{m(n_i-1)}}
\geq \prod_{i=1}^r \frac{1}{p^{m(n_i-1)}}
$$
whenever $m'>m$ and $\tilde \u \in (\ZZ/p^{m'}\ZZ)^s$ is such that 
$\tilde \u \equiv \u \Mod{p^{m}}$.
For any given $m'>m$, there are $p^{(m'-m)s}$ admissible choices for $\tilde \u$.
Note that $m > v_p(Q)$.
Thus,
$$
\beta_p 
\geq \lim_{m' \to \infty} \frac{1}{p^{m's}} 
\sum_{\substack{ \tilde \u \in (\ZZ/p^{m'}\ZZ)^s 
                \\ \tilde \u' \equiv \u \Mod{p^{m'}}}}
\prod_{i=1}^r \frac{\rho_i(p^{m'}, f_i(\tilde \u))}{p^{m'(n_i-1)}}
\geq \frac{1}{p^{m(n_1 + \dots + n_r + s -r)}}
>0,
$$
which completes the proof. 
\end{proof}

\section{$R^*_S$ in arithmetic progressions} \label{s:progressions}

This section contains two lemmas about the mean value of $R^*_{S}$ in 
arithmetic progressions.
These will be required in the proof of Theorem \ref{p:nilsequences} in 
Section \ref{s:proof}.

\begin{lemma}\label{l:R-mean-value}
Let $S$ be a finite set of primes and let $R^*_S$ be the corresponding 
restriction of the representation function. 
Let $N$ and $q$ be positive integers such that $p|q$ for every prime $p<w(N)$ and 
such that $v_p(q)\not=1$ for all $p \not\in S$. 
Suppose further that $A \in \{1,\dots,q \}$ is an integer such that
$0 \leq v_p(A) \leq 1$ whenever $p|q$ and $p \not\in S$.
Let $\epsilon \in \{\pm\}$ and let $\kappa^\epsilon$ be the constant that 
appears in \cite[Lemma 6.1]{bm}.
Then, provided $N$ is sufficiently large,
\begin{equation} \label{eq:R'-sum}
\sum_{\substack{0< \epsilon m \leq x \\ m \equiv A \Mod{q}}} R^*_S(m)
= 
 \kappa^\epsilon x \frac{\rho(q,A)}{q^{n}}
 \prod_{p\nmid q} \bigg(1-\frac{\rho(p^2,0)}{p^{2n}}\bigg)
 + O(qx^{1-\frac{1}{20n}}).
\end{equation} 
\end{lemma}
\begin{proof}
We shall deduce this result from \cite[Lemma 6.1]{bm}, which states that 
\begin{equation} \label{eq:R-sum}
\sum_{\substack{0< \epsilon m \leq x \\ m \equiv A' \Mod{q'}}} R(m)
= \frac{\rho(q',A')}{{q'}^{n}} \kappa^\epsilon x
 + O(q'x^{1-1/n}),
\end{equation} 
for any positive integer $q'$, any $A' \in \ZZ$ and 
$\epsilon \in \{\pm\}$.
Since
\begin{align*}
 R^*_S(m) = \sum_{\substack{d^2 | m \\ \gcd(d,q)=1}} \mu(d) R(m)
\end{align*}
for $m \equiv A \Mod{q}$, we have
\begin{align} \label{eq:R'-main-dec}
\sum_{\substack{0< \epsilon m \leq x \\ m \equiv A \Mod{q}}} R^*_S(m)
&= 
\sum_{\substack{d \leq x^{1/2}\\ \gcd(d,q)=1}}
\mu(d)
\sum_{\substack{0< \epsilon m \leq x 
               \\ m \equiv A \Mod{q} 
               \\ m \equiv 0 \Mod{d^2}}} 
R(m). 
\end{align}
If $d$ is sufficiently small, then the inner sum may be evaluated by 
means of \eqref{eq:R-sum}.
Together with the Chinese remainder theorem we obtain
\begin{align}\label{eq:R-small-d}
 \nonumber
&\sum_{\substack{d < x^{1/10n} \\ \gcd(d,q)=1}} 
  \mu(d)
 \sum_{\substack{0< \epsilon m \leq x
                 \\m \equiv A \Mod{q}
                 \\ m \equiv 0 \Mod{d^2}}} 
   R(m) \\
 \nonumber
&=   
  \sum_{\substack{d \leq x^{1/10n} \\ \gcd(d,q)=1}} 
   \left(
   \mu(d)
   \frac{\rho(d^2,0)}{d^{2n}} 
   \frac{\rho(q,A)}{{q}^{n}} 
   \kappa^\epsilon x
    + O(d^2 q x^{1-\frac{1}{n}})
   \right) \\
&=   
  \sum_{\substack{d \leq x^{1/10n} \\ \gcd(d,q)=1}}
   \mu(d)
   \frac{\rho(d^2,0)}{d^{2n}} 
   \frac{\rho(q,A)}{{q}^{n}} 
   \kappa^\epsilon x
   + O(q x^{1-\frac{1}{n} + \frac{3}{10n}}) .
\end{align}
We aim to extend the summation in $d$ to all positive integers that are 
co-prime to $q$.
By multiplicativity of $\rho$ we deduce from \eqref{eq:rho-p^2-0} that
\begin{equation}\label{eq:rho(d^2,0)-bound}
 \mu^2(d) \frac{\rho(d^2,0)}{d^{2n}} 
\leq \frac{(C2^{n})^{\omega(d)}}{d^2}
\ll_{C,n,\eps} d^{-2 + \eps}
\end{equation}
for any $\eps >0$.
Since
\begin{align*}
 \sum_{d> x^{1/10n}} d^{-2 +\eps}
 \ll x^{\frac{1}{10n}(-2 + \eps +1)}
 \ll x^{-\frac{1}{20n}}
\end{align*}
for $\eps$ sufficiently small, \eqref{eq:R-small-d} is seen to equal
\begin{align}\label{eq:1.8}
\nonumber
&
   \kappa^\epsilon x
   \frac{\rho(q,A)}{{q}^{n}}
   \Bigg(
   \sum_{\substack{d \geq 1 \\ \gcd(d,q)=1}}
   \mu(d)
   \frac{\rho(d^2,0)}{d^{2n}} 
   + O(x^{-\frac{1}{20n}})
   \Bigg)
   + O(q x^{1-\frac{1}{n} + \frac{3}{10n}}) \\
&=   
   \kappa^\epsilon x
   \frac{\rho(q,A)}{{q}^{n}} 
   \prod_{p\nmid q}
   \bigg(1- \frac{\rho(p^2,0)}{p^{2n}}\bigg) 
   + O(x^{1-\frac{1}{20n}})
   + O(q x^{1-\frac{1}{n} + \frac{3}{10n}}),
\end{align}
where we made use of the trivial bound $\rho(q,A) \ll q^n$.

In order to bound the tail of the summation in $d$ from 
\eqref{eq:R'-main-dec}, we consider a fixed square-free integer $d$ with 
$\gcd(d,q)=1$ and let $d=p_1 \dots p_k$ be its prime factorisation.
As shown in \cite[Lemma 8.1]{bm}, the representation function $R$ may be 
uniformly bounded above by a multiplicative function of constant average order. 
More precisely, we have $R(m) \ll r_K(|m|)$, for all $m \not= 0$, where $r_K$ is 
the multiplicative function describing the Dirichlet coefficients of the Dedekind 
zeta function of $K$, i.e. $\zeta_K(s)= \sum_n r_K(n)n^{-s}$. 
Invoking two basic properties of $r_K$, namely that 
$r_K(m) \ll \tau(m)^n$ and that $\sum_{m \leq x} r_K(m) \ll x$, c.f.\ 
\cite[(2.8) and (2.10)]{bm}, we deduce
\begin{align*}
  \sum_{\substack{m \equiv A \Mod{q}
                 \\ m \equiv 0 \Mod{d^2}
                 \\ 0< \epsilon m \leq x}} 
  R(m)
&\ll \sum_{(b_1,\dots,b_k) \in \NN_0^k}
  r_K(d^2 p_1^{b_1} \dots p_k^{b_k})
 \sum_{0 < m \leq x/(d^2 p_1^{b_1} \dots p_k^{b_k} )} 
  r_K(m) \\
&\ll \sum_{\substack{(b_1,\dots,b_k) \in \NN_0^k
                    \\ p_1^{b_1} \dots p_k^{b_k} \leq x}} 
 \tau(d^2 p_1^{b_1} \dots p_k^{b_k})^n
 \frac{x}{d^2 p_1^{b_1} \dots p_k^{b_k}} \\
&\ll x \prod_{i=1}^k 
 \sum_{b_i \geq 2} \frac{\tau(p_i^{b_i})^n}{p_i^{b_i}} 
 \ll x \prod_{i=1}^k 
 \sum_{b_i \geq 2} \Big(\frac{2^n}{p_i}\Big)^{b_i}
\ll x \prod_{i=1}^k \frac{2^{2n+1}}{{p_i}^2}
\ll x \frac{\tau(d^2)^{n+1}}{d^2}.
\end{align*}
Thus, for any $C_0>2$ we have
\begin{align}\label{eq:tail-est}
\nonumber
 \Bigg|
 \sum_{\substack{x^{1/C_0} \leq d \leq x^{1/2} \\ \gcd(d,q)=1}} 
\mu(d)
 \sum_{\substack{m \equiv A \Mod{q}\\ m \equiv 0 \Mod{d^2}\\ 
       0 < \epsilon m \leq x}} 
  R(m)
 \Bigg|
&\leq
 \sum_{\substack{x^{1/C_0} \leq d \leq x^{1/2} \\ \gcd(d,q)=1\\ \mu^2(d)=1}}
 \sum_{\substack{m \equiv A \Mod{q}\\ m \equiv 0 \Mod{d^2}\\ 
       0 < \epsilon m \leq x}} 
   R(m) \\
\nonumber
&\ll
 x 
 \sum_{\substack{x^{1/C_0} \leq d \leq x^{1/2} \\ \gcd(d,q)=1}}
 \frac{\tau(d^2)^{n+1}}{d^2} \\
\nonumber
&\ll_{\eps}
 x 
 \sum_{\substack{x^{1/C_0} \leq d \leq x^{1/2} \\ \gcd(d,q)=1}}
 d^{-2+\eps} \\
&\ll_{\eps} x^{1-\frac{1-\eps}{C_0}}.
\end{align}
Combining this estimate for $C_0=10n$ with \eqref{eq:R-small-d} and \eqref{eq:1.8}
completes the proof.
\end{proof}

Our next aim is to establish the `major arc estimate' that is required 
in order to deduce Theorem \ref{p:nilsequences} from a 
non-correlation estimate that only involves sufficiently \emph{equidistributed} 
polynomial nilsequences.
The following lemma corresponds to \cite[Lemma 6.2]{bm} and shows that the 
$W$-tricked version of $R^*_S$ has constant average values on certain 
subprogressions.
Recall the definition of $\cA(R^*_S, N)$ from \eqref{eq:def-A}.
\begin{lemma}[Major arc estimate]\label{l:major-arc}
Let $\epsilon \in \{\pm\}$ and let $N>0$ be an integer.
Suppose $A \in \mathcal \cA(R^*_S, N)$ and let $q_0$ be a $w(N)$-smooth number.
Let $x,x' \in \ZZ_{>0}$ be parameters that satisfy $x \asymp x'$.
Then
\begin{align*}
  \frac{W(N)}{x} 
  \sum_{\substack{m \equiv A \Mod{W}\\ 0 < \epsilon m \leq x}}
  R^*_S(m)
 = \frac{W(N)q_0}{x'}
  \sum_{\substack{m \equiv A + Wq_1 \Mod{Wq_0}\\ 0 < \epsilon m \leq x'}}
  R^*_S(m)
 + O(q_0^2W(N)^2x^{-\frac{1}{20n}})
\end{align*}
for all integers $q_1$ and all sufficiently large $N$.
\end{lemma}

\begin{proof}
We shall employ the lifting result \cite[Lemma 3.4]{bm}, 
which in our context states the following.
Let $m\geq 1$, $A' \not= 0$ and assume that 
$$v_p(A')+v_p(n)< \frac{m}{2}.
$$
Then we have 
\begin{equation} \label{eq:rho-lifting}
\frac{\rho(p^m,A')}{p^{m(n-1)}} = 
\frac{\rho(p^{m+1},A'+kp^m)}{p^{(m+1)(n-1)}},
\end{equation} 
uniformly for $k\in \ZZ/p\ZZ$.

Since the definition of $\cA(R^*_S, N)$ guarantees that all of the 
above assumptions are satisfied, we deduce that
\begin{equation} \label{eq:lifting}
  \frac{\rho(W,A)}{(W)^{n-1}} 
= \frac{\rho(Wq_0,A + Wq_1)}{(Wq_0)^{n-1}}.
\end{equation}
The lemma now follows from an application of Lemma \ref{l:R-mean-value} to each 
of the two sums over $R^*_S$ from the statement, combined with an application 
of the identity \eqref{eq:lifting}.
\end{proof}

\section{Polynomial subsequences of multiparameter nilsequences}
\label{s:technical}
In this section we recall some of the background on equidistribution of 
multiparameter polynomial nilsequences and prove several technical results that 
analyse to what extent equidistribution properties are preserved when passing to 
certain subsequences or families of subsequences.

Throughout what follows, $[x]$ denotes the set of integers 
$\{1,\dots,\lfloor x \rfloor \}$. 
We shall be working with the quantitative notion of 
equidistribution that was introduced by Green and Tao in 
\cite[Def.\ 8.5]{GT-polyorbits}:
\begin{definition}[Quantitative equidistribution]
Let $G/\Gamma$ be a nilmanifold equipped with Haar measure and let $\delta > 0$.
A finite sequence 
$$(g(\mathbf{n})\Gamma)_{\mathbf{n} \in [N_1] \times \dots \times [N_t]}$$ 
taking values in $G/\Gamma$ is called \emph{$\delta$-equidistributed} if
$$
\left|
\frac{1}{|N_1| \dots |N_t|}
\sum_{\mathbf{n} \in [N_1] \times \dots \times [N_t]} F(g(\mathbf{n})\Gamma)
- \int_{G/\Gamma} F
\right|
\leq \delta \|F\|_{\mathrm{Lip}}
$$
for all Lipschitz functions $F:G/\Gamma \to \CC$.
It is said to be \emph{totally $\delta$-equidistributed} if moreover
$$
\left|
\frac{1}{|P_1|\dots |P_t|}
\sum_{\mathbf{n} \in P_1 \times \dots \times P_t} F(g(\mathbf{n})\Gamma) 
- \int_{G/\Gamma} F
\right|
\leq \delta \|F\|_{\mathrm{Lip}}
$$
for all Lipschitz functions $F:G/\Gamma \to \CC$ and for all collections of 
arithmetic progressions $P_i \subset \{1, \dots, N_i\}$
of length $|P_i| \geq \delta N_i$ for $1 \leq i \leq t$.
\end{definition}

The most relevant measure in the analysis of quantitative equidistribution of 
polynomial sequences are the smoothness norms. 
These, too, were introduced in \cite{GT-polyorbits}; see also \cite{GT-erratum}.

\begin{definition}[Smoothness norms]
Let $f: \ZZ^t \to \RR/\ZZ$ be a polynomial of degree $d$ and suppose  
$$
f(n_1, \dots, n_t) 
= 
\sum_{\substack{(i_1, \dots, i_t) \in \ZZ_{\geq0}^t\\ i_1 + \dots + i_t \leq d}}
\beta_{i_1, \dots, i_t} n_1^{i_1} \dots n_t^{i_t}.
$$
Then
$$
\|f\|_{C_*^{\infty}[N_1] \times \dots \times [N_t]} 
:=
\sup_{(i_1, \dots, i_t) \not= \0}
N_1^{i_1} \dots N_t^{i_t} 
\|\beta_{i_1, \dots, i_t}\|.
$$
\end{definition}

Finally, recall that a continuous additive homomorphism $\eta: G \to \RR/\ZZ$ is 
called a horizontal character if it annihilates $\Gamma$.
The equidistribution properties of multiparameter nilsequences can be analysed 
through horizontal characters on $G/\Gamma$ via a theorem of Green and Tao, 
\cite[Theorem 8.6]{GT-polyorbits}, which we state below.
See \cite{GT-erratum} for a proof.

\begin{theorem}[Green-Tao \cite{GT-polyorbits, GT-erratum}] \label{t:GT-eq}
 Let $0 < \delta < 1/2$ and let $m, t ,d, N_1, \dots, N_t \geq 1$ be 
positive integers.
 Suppose that $G/\Gamma$ is an $m$-dimensional nilmanifold equipped with a 
 $\frac{1}{\delta}$-rational Mal'cev basis $\mathcal{X}$ adapted to some 
 filtration $G_{\bullet}$ of degree $\ell$, and that 
 $g \in \mathrm{poly}(\ZZ^t, G_{\bullet})$.
 Then either 
 $(g(\mathbf{n})\Gamma)_{\mathbf{n} \in [N_1] \times \dots \times [N_t]}$ is
 $\delta$-equidistributed, or else there is some horizontal character $\eta$ with
 $0 < |\eta| \ll \delta^{-O_{\ell,m,t}(1)}$ such that
 $$\| \eta \circ g \|_{C^{\infty}[N_1] \times \dots \times [N_t]} 
 \ll \delta^{-O_{\ell,m,t}(1)}.$$
\end{theorem}

In the case of polynomial nilsequences, the quantitative notions of 
equidistribution and total equidistribution are equivalent, with polynomial 
dependence in the equidistribution parameter.
The following lemma handles the non-trivial direction of this equivalence.

\begin{lemma}\label{l:totaleq/eq}
 Suppose $\delta: \NN \to (0,1/2)$ is such that $\delta(x)^{-T} \ll_T x$ for 
 $T>0$.
 Let $m, \ell, t$ and $N$ be positive integers and suppose that $G/\Gamma$ is an 
 $m$-dimensional nilmanifold equipped with a 
 $\frac{1}{\delta(N)}$-rational Mal'cev basis $\mathcal{X}$ adapted to some 
 filtration $G_{\bullet}$ of degree $\ell$.
 Then there is a constant $1 \leq C \ll_{\ell,m} 1$ such that the following holds.
 
 Let $E>C$ and let $g \in \mathrm{poly}(\ZZ^t, G_{\bullet})$. 
 Suppose that the finite sequence 
 $(g(\bn)\Gamma)_{\bn \in [N]^t}$ is $\delta(N)^E$-equidistributed. 
 Then  $(g(\bn)\Gamma)_{\bn \in [N]^t}$ is totally 
 $\delta(N)^{E/C}$-equidistributed, provided $N$ is sufficiently large.
\end{lemma}
\begin{proof}
 This is a rather straightforward generalisation of the computation carried out 
in the proof of \cite[Lemma 6.2]{lmm}.
\end{proof}
We note aside that the factorisation theorem for nilsequences 
\cite[Theorem 1.19]{GT-polyorbits} will allow us to assume that $E$ is 
sufficiently large for the condition $E>C$ of the above lemma to be 
satisfied at all instances we will make use of it.

The multiparameter nilsequences that will be most relevant to the proof of 
Theorem~\ref{p:nilsequences} are those that arise as the composition 
$g \circ P$ of a polynomial $P \in \ZZ[X_1, \dots, X_t]$ and a one-parameter 
nilsequence $g$.

\begin{lemma}
Suppose $\delta: \NN \to (0,1/2)$ is a function that satisfies
$\delta(x)^{-T} \ll_T x$ for all $T > 0$.
Let $m, \ell, t$ and $N$ be positive integers and suppose that $G/\Gamma$ 
is an $m$-dimensional nilmanifold equipped with a $\frac{1}{\delta(N)}$-rational 
Mal'cev basis $\mathcal{X}$ adapted to some filtration $G_{\bullet}$ of degree 
$\ell$.
Let $P \in \ZZ[X_1,\dots,X_t]$ be a homogeneous polynomial of degree $t$, fixed 
once and for all, and let all implied constants be allowed to depend on the 
coefficients of $P$ in any way.
Then there is a constant $1 \leq C \ll_{t,\ell} 1$ such that the following 
holds.

Let $E>C$ and let $g \in \mathrm{poly}(\ZZ^t, G_{\bullet})$. 
Suppose that  
$(g(n) \Gamma)_{n\leq N}$ is totally $\delta(N)^E$-equidistributed.
Then
$$
(g(P(n_1\dots,n_t) \Gamma))_{
(n_1\dots,n_t) \in [N^{1/t}] \times \dots \times[N^{1/t}]}
$$
is totally $\delta(N)^{E/C}$-equidistributed whenever $N$ is sufficiently large.
\end{lemma}

\begin{proof}
In order to apply Lemma \ref{l:totaleq/eq} and Theorem \ref{t:GT-eq} 
to the sequence
$(g(P(\bn)))_{\bn \in \ZZ^t}$, we require first of all a filtration 
$G'_{\bullet}$ with respect to which 
$(g(P(\bn)))_{\bn \in \ZZ^t}$ is a polynomial sequence.
It follows from \cite[Lemma 6.7]{GT-polyorbits} that $g$ has a representation 
of the form
$g(n) = a_1^{P_1(n)}\dots a_k^{P_k(n)}$ where 
$a_1, \dots, a_k \in G$ and where $P_1, \dots, P_k \in \ZZ[X]$ are polynomials of 
degree at most $\ell$.
Thus, 
$g(P(\bn))= a_1^{P_1\circ P(\bn)}\dots a_k^{P_k \circ P(\bn)}$, where each 
polynomial $P_i \circ P$  
has degree at most $\ell' = \ell \cdot \max_{1\leq i \leq k}(\deg P_i)$.
Define a filtration $G'_{\bullet}$ by setting $G'_j=G_{\lceil j/\ell' \rceil}$ 
for $0\leq j \leq \ell' \ell$.
Then it is immediate (c.f.\ the discussion following Lemma 6.7 in 
\cite{GT-polyorbits}) that
$(a_i^{P_i \circ P(\bn)})_{\bn \in \ZZ^t}$ belongs to 
$\mathrm{poly}(G'_{\bullet}, \ZZ^t)$
for each $1\leq i \leq k$.
By Leibman's theorem~\cite{Leibman} (see \cite{GT-polyorbits} for a different 
proof), the set $\mathrm{poly}(G'_{\bullet}, \ZZ^t)$ forms a group.
Thus it follows that 
$(g(P(\bn)))_{\bn \in \ZZ^t} \in \mathrm{poly}(\ZZ^t,G'_{\bullet})$.

Finally observe that the given Mal'cev basis $\mathcal{X}$ is a Mal'cev basis 
adapted to $G'_{\bullet}$ as well.
Indeed, part (ii) of \cite[Definition 2.1]{GT-polyorbits} follows immediately from 
the corresponding statement for $G_{\bullet}$ since  
$\{G'_i: 0\leq i \leq s\} \subseteq \{G_i: 0\leq i \leq s\}$. 

We are now in the position to start with the proof of the lemma.
Suppose $B\geq1$ and that
$$
(g(P(n_1\dots,n_t) \Gamma))_{
(n_1\dots,n_t) \in [N^{1/t}]^t}
$$
fails to be totally $\delta(N)^{B}$-equidistributed.
Then, by Lemma \ref{l:totaleq/eq} there is a constant $1 \leq C_1 \ll_{\ell,m} 1$ 
such that the above sequence also fails to be 
$\delta(N)^{C_1B}$-equidistributed.
Thus, Theorem \ref{t:GT-eq} implies that there is a non-trivial horizontal 
character
$\eta: G/\Gamma \to \CC$ of modulus $|\eta| \ll \delta(N)^{-O_{\ell,m,t}(B)}$ 
such that
$$\|\eta \circ g \circ P\|_{C_*^{\infty}[N^{1/t}]^t} 
 \ll \delta(N)^{-O_{\ell,m,t}(B)}.$$
Writing
$$
P(n)^j= 
\sum_{\substack{i_1, \dots, i_t \geq 0 \\ i_1 + \dots + i_t = tj}}
\gamma_{i_1, \dots, i_t}^{(j)} n_1^{i_1} \dots n_t^{i_t}
\qquad (1 \leq j \leq \ell),
$$
then all coefficients $\gamma_{i_1, \dots, i_t}^{(j)}$ are bounded.
If further
$$\eta \circ g (n) = \sum_{i=j}^{\ell} \beta_j n^j,$$
then 
\begin{equation}\label{eq:gP-sup}
\| \eta \circ g \circ P \|_{C_*^{\infty}[N^{1/t}]^t} = 
\sup_{\substack{1\leq j \leq \ell \\ i_1 + \dots + i_t = tj}} 
N^j \|\beta_j \gamma_{i_1, \dots, i_t}^{(j)} \| 
\ll \delta(N)^{-O_{\ell,m,t}(B)}. 
\end{equation}
Let $\gamma$ be the least common multiple of all non-zero 
coefficients $\gamma_{i_1, \dots, i_t}^{(j)}$ and 
observe that $\gamma \ll_{\ell,m,t} 1$.
Since $\delta(x)^{-T}\ll_T x$, we have 
$$\|\beta_j \gamma_{i_1, \dots, i_t}^{(j)} \|
\ll \delta(N)^{-O_{\ell,t,m}(B)}N^{-j} = o_{\ell,t,m}(1)$$ 
whenever $i_1 + \dots + i_t = tj$ and $1 \leq j \leq \ell$.
Hence, provided $N$ is sufficiently large, we have 
$\|A \beta_j \gamma_{i_1, \dots, i_t}^{(j)} \| 
= A \|\beta_j \gamma_{i_1, \dots, i_t}^{(j)} \|$
for any positive real $A \ll_{\ell,t,m} 1$.
In particular, we have
$\|\beta_j\gamma\|
\ll_{\ell,t,m} \|\beta_j \gamma_{i_1, \dots, i_t}^{(j)} \|$ whenever
$\gamma_{i_1, \dots, i_t}^{(j)}$ is non-zero.
Since for every $j \in \{1,\dots,\ell\}$ at least one of the coefficients 
$\gamma_{i_1, \dots, i_t}^{(j)}$ of $P^j$ is non-zero, the above and 
\eqref{eq:gP-sup} imply that
$$
\|\gamma \eta \circ g\|_{C_*^{\infty}[N]}
=
\sup_{1\leq j \leq \ell } 
N^j \|\beta_j \gamma \| \ll \delta^{-O_{\ell,m,t}(B)},
$$
provided $N$ is sufficiently large.
Since $\gamma \eta$ is a non-trivial horizontal character of modulus 
$|\gamma \eta| \ll \delta^{-O_{\ell,m,t}(B)}$, we deduce that
(cf.\ \cite[Propositions 14.2 and 14.3]{lmr}) there is a constant 
$C_2 \asymp_{\ell,m,t} 1$ such that $(g(n)\Gamma)_{n\leq N}$ fails to be totally 
$\delta(N)^{C_2B}$-equidistributed.
Choosing $C = \max (1,C_2)$, the result follows for every $E>C$ by setting 
$B=E/C$.
Indeed, when $E=CB \geq C_2B$, then the above conclusion that 
$(g(n)\Gamma)_{n\leq N}$ fails to be totally $\delta(N)^{C_2B}$-equidistributed 
contradicts the assumption that this sequence is totally 
$\delta(N)^{E}$-equidistributed.
\end{proof}

Our next aim is to extend the above lemma in a way that allows us to replace the 
homogeneous polynomial $P$ by an inhomogeneous polynomial of the form
$$\x \mapsto \frac{P(Wq\x + \y) - A'}{Wq},$$
where $W=W(N)$ is given by \eqref{def:W}, where $q \in \NN$,
where $\y \in \ZZ^t$ is such that $0 \leq y_i < Wq$ for $1\leq i \leq t$, 
and where $A'$ is such that $P(\y) \equiv A' \Mod{Wq}$ and $|A'| < Wq$.

\begin{lemma} \label{l:poly-subsequences}
Let $N$ be a positive integer and suppose $T=T(N)$ satisfies 
$N^{1 - \eps} \ll_{\eps} T \leq N$ for all $\eps > 0$.
Let $\delta: \NN \to (0,1/2)$ be a function that satisfies 
$\delta(x)^{-1} \ll_{\eps} x^{\eps}$ for all $\eps > 0$.
Let $m,\ell$ and $t$ be positive integers and suppose that $G/\Gamma$ is 
an $m$-dimensional nilmanifold equipped with a $\frac{1}{\delta(N)}$-rational 
Mal'cev basis $\mathcal{X}$ adapted to some filtration $G_{\bullet}$ of degree 
$\ell$. 
Let $g \in \mathrm{poly}(\ZZ, G_{\bullet})$ be any polynomial sequence, and let 
$P \in \ZZ[X_1, \dots, X_t]$ be a fixed homogeneous polynomial of degree $t$. 
All implied constants are allowed to depend on the coefficients of $P$ in any 
way. 
Suppose $S:\NN \to \NN$ satisfies $S(x) \ll_{\eps} x^{\eps}$ for all $\eps>0$.

Then there is a constant $1 \leq C\ll_{m,\ell,t} 1$ such that the following holds.
Let $E>C$ and suppose that for every $w(N)$-smooth integer $\tilde q \leq S(N)$ 
the sequence $(g(\tilde qn)\Gamma)_{n \leq T/\tilde q}$ is totally 
$\delta(N)^E$-equidistributed.
Further, let $q>0$ be a $w(N)$-smooth integer that satisfies the bound 
$(Wq)^{t\ell^2} \leq S(N)$ where $W=W(N)$.
Then, provided $N$ and $T$ are sufficiently large,
$$
\bigg(g\bigg(\frac{P(Wq\x+\y) - A'}{Wq}\bigg)\Gamma\bigg)_
{\x \in \left[\left(\frac{T}{(Wq)^{t-1}}\right)^{1/t}\right]^t}
$$
is a totally $\delta(N)^{E/C}$-equidistributed sequence for every choice of
$\y \in \ZZ^t$ such that $0 \leq y_i < Wq$ for $1\leq i \leq t$ and for 
$A' \in \ZZ$ such that $P(\y) \equiv A' \Mod{Wq}$ and $|A'| < Wq$.
\end{lemma}

\begin{proof}
Let us write
$$
\tilde g(\x) = g\bigg(\frac{P(Wq\x+\y) - A'}{Wq}\bigg).
$$
Similarly as in the previous proof, there is a refinement $G'_{\bullet}$ of 
the filtration $G_{\bullet}$ such that the new filtration is adapted to the basis 
$\mathcal{X}$ and its degree is of order $O_{\ell,t}(1)$, and such that
$(\tilde g(\x) \Gamma)_{\x \in \ZZ^t} \in \mathrm{poly}(G'_{\bullet}, \ZZ^t)$.

Let $B>1$ and suppose that
$$
\Big(\tilde g(\x) \Gamma\Big)_
{\x \in \left[\left(\frac{T}{(Wq)^{t-1}}\right)^{1/t}\right]^t}
$$
fails to be totally $\delta(N)^{B}$-equidistributed.
Then, as in the proof of the previous lemma, Lemma \ref{l:totaleq/eq} and Theorem 
\ref{t:GT-eq} imply that there is a non-trivial horizontal character 
$\eta:G/\Gamma \to \CC$ such that
$|\eta| \ll \delta(N)^{-O_{\ell,m,t}(B)}$ and
$$
\|\eta \circ \tilde g\|_{C_*^{\infty} 
\left[\left(\frac{T}{(Wq)^{t-1}}\right)^{1/t}\right]^t}
\ll \delta^{-O_{m,\ell,t}(B)}.
$$
Suppose
$$
P(\bn)^j= 
\sum_{\substack{i_1, \dots, i_t \geq 0 \\ i_1 + \dots + i_t = tj}}
\gamma^{(j)}_{i_1, \dots, i_t} n_1^{i_1} \dots n_t^{i_t},
\qquad (1 \leq j \leq \ell)
$$
and note that for each $j$ at least one of the coefficients 
$\gamma^{(j)}_{i_1, \dots, i_t}$ is non-zero.
Furthermore, suppose
$$\eta \circ g (n) = \sum_{j=0}^{\ell} \beta_j n^j,$$
where $\beta_j\not= 0$ for at least one value $j>0$. 

We proceed by analysing the coefficients of the polynomial map 
$\x \mapsto \eta \circ g(\frac{P(Wq\x+\y) - A'}{Wq})$.
To begin with, observe that
$$
\frac{P(Wq\x+\y) - A'}{Wq}
= (Wq)^{t-1}P(\x) + P'(\x),
$$
for some polynomial $P'$ of degree $t-1$ with coefficients of order 
$O((Wq)^{t-1})$.
Inserting this information into the above expression for $\eta \circ g$, we obtain
\begin{align*}
\eta \circ g \bigg(\frac{P(Wq\x+\y) - A'}{Wq}\bigg)
&= 
\sum_{j=0}^{\ell} \beta_j
\sum_{\substack{i_1, \dots, i_t \geq 0 \\ i_1 + \dots + i_t = tj}}
(Wq)^{(t-1)j}
\gamma^{(j)}_{i_1, \dots, i_t} x_1^{i_1} \dots x_t^{i_t}\\
&+
\sum_{j=0}^{\ell} \beta_j
\sum_{\substack{i_1, \dots, i_t \geq 0 \\ i_1 + \dots + i_t \leq tj-1}}
c^{(j)}_{i_1, \dots, i_t}
x_1^{i_1} \dots x_t^{i_t},
\end{align*}
where $|c^{(j)}_{i_1, \dots, i_t}| \ll (Wq)^{(t-1)j}$.
If $\eta \circ \tilde g$ has the representation
$$
\eta \circ \tilde g(\x)
= \sum_{\substack{i_1, \dots, i_t \geq 0 \\ i_1 + \dots + i_t \leq t\ell}}
\alpha_{i_1, \dots, i_t} x_1^{i_1} \dots x_t^{i_t},
$$
then
$$
\sup_{i_1 + \dots + i_t \leq t\ell}
\bigg(\frac{T}{(Wq)^{t-1}}\bigg)^{\frac{i_1 + \dots + i_t}{t}}
\|\alpha_{i_1, \dots, i_t} \| \ll \delta^{-O_{\ell,m,t}(B)},
$$
or, in other words,
\begin{equation} \label{eq:smoothness-bd}
\|\alpha_{i_1, \dots, i_t} \| 
\ll \delta^{-O_{\ell,m,t}(B)}
(Wq)^{(i_1 + \dots + i_t)\frac{t-1}{t}}
T^{-(i_1 + \dots + i_t)/t}
\end{equation}
holds uniformly for all admissible tuples  $(i_1, \dots, i_t)$.
Since $W(N)q \ll_{\eps} T^{\eps}$ and $\delta^{-1}(x)\ll_{\eps} x^{\eps}$, we in 
fact have the following `graded' bounds in terms of the value 
$j = (i_1+ \dots + i_t)/t$:
\begin{equation} \label{eq:smoothness-bd-1}
\|\alpha_{i_1, \dots, i_t} \| 
\ll_{\eps} 
T^{-\frac{i_1 + \dots + i_t}{t} + \eps + o(1)}.
\end{equation}
Let $\gamma$ be, as before, the least common multiple of non-zero coefficients 
$\gamma_{i_1, \dots, i_t}^{(j)}$.
We aim to deduce from \eqref{eq:smoothness-bd} and \eqref{eq:smoothness-bd-1} 
bounds with a graded decay for the terms $\|\beta_{j} q_j \|$, $1 \leq j \leq 
\ell$, too, where the $q_j$ are small integers compared to $S(N)$ and 
$w(N)$-smooth; in fact, they will take the form
$$q_j:= (Wq)^{(t-1)(\ell + (\ell-1)  + \dots  +j)} \gamma^{1+\ell-j}.$$
Note that
$$\alpha_{i_1, \dots, i_t} 
= \beta_d (Wq)^{(t-1)\ell} \gamma^{(\ell)}_{i_1, \dots, i_t}$$
whenever $i_1 +  \dots + i_t = t\ell$.
By \eqref{eq:smoothness-bd}, this immediately yields
$$\| \beta_{\ell} (Wq)^{(t-1)\ell} \gamma \| 
\ll \delta^{-O_{\ell,m,t}(1)}
(Wq)^{(t-1)\ell} T^{-\ell}.$$
More generally, we have
$$
\alpha_{i_1, \dots, i_t} 
= \beta_j (Wq)^{(t-1)j} \gamma^{(j)}_{i_1, \dots, i_t}
+ \sum_{k>j} \beta_k c^{(k)}_{i_1, \dots, i_t}
$$
when $i_1 + \dots + i_t = tj$.
Multiplying through by $q_{j+1}$, this identity allows us to employ a 
downwards-inductive argument, taking advantage of the graded decay bounds that 
can be assumed inductively for $\|\beta_k q_k\|$ with $k>j$.
Indeed, by applying \eqref{eq:smoothness-bd} 
to $\alpha_{i_1, \dots, i_t}$ and the induction hypotheses to 
$\beta_k q_k$ for $k>j$, we deduce that
$$\| \beta_j q_j \| 
\ll \delta^{-O_{\ell,m,t}(B)}
q_j T^{-j}.
$$
It follows that
$$
\sup_{1\leq j \leq \ell} (T/q_{\ell})^{j} \|q_{\ell}^j \beta_j \| \ll 
\delta^{-O_{\ell,m,t}(B)}.
$$
If $N$ and $T$ are sufficiently large, then Lemma \ref{l:totaleq/eq} 
and Theorem \ref{t:GT-eq} imply that there is $C_1 \asymp_{\ell,m,t} 1$ such that
$g(q_{\ell} n)_{n \leq T/q_{\ell}}$ fails to be totally 
$\delta(N)^{C_1B}$-equidistributed.
Setting $C=\max(C_1,1)$, the result follows for every $E>C$ by choosing 
$B=E/C$.
Indeed, if $E=BC>BC_1$ and if $g$ satisfies all hypotheses from the 
statement, then, in particular, $g(q_{\ell} n)_{n \leq T/q_{\ell}}$ is totally 
$\delta(N)^{E}$-equidistributed.
This is a contradiction and shows that
$$
\Big(\tilde g(\x) \Gamma\Big)_
{\x \in \left[\left(\frac{T}{(Wq)^{t-1}}\right)^{1/t}\right]^t}
$$
is in fact totally $\delta(N)^{E/C}$-equidistributed.
\end{proof}

The next lemma is in spirit closely related to the previous one. 
It shows that the assumptions the previous lemma makes on the polynomial 
sequence $g$ imply that these assumptions, with a slightly different constant 
$E$, are also met by any sequence of the form $g \circ L$ for certain linear 
polynomials $L$.
This result will allow us to replace $g$ by $g \circ L$ in the conclusion of Lemma
\ref{l:poly-subsequences} and thus to easily deal with a necessary restriction  
to subprogressions in Section \ref{s:proof}.
We note aside that this result generalises to higher degree polynomials.

\begin{lemma} \label{l:weaker-assumptions}
Let $N$ and $T$ be positive integers and suppose that $T=T(N)$ satisfies
$N^{1 - \eps} \ll_{\eps} T \leq N$ for all $\eps > 0$.
Let $\delta: \NN \to (0,1/2)$ be a function that satisfies 
$\delta(x)^{-1} \ll_{\eps} x^{\eps}$ for all $\eps > 0$.
Let $m_G$ and $\ell$ be positive integers and suppose that $G/\Gamma$ is 
an $m_G$-dimensional nilmanifold equipped with a $\frac{1}{\delta(N)}$-rational 
Mal'cev basis $\mathcal{X}$ adapted to some filtration $G_{\bullet}$ of degree 
$\ell$. 
Let $g \in \mathrm{poly}(\ZZ, G_{\bullet})$ be any polynomial sequence and let
$S:\NN \to \NN$ be a function such that $S(x) \ll_{\eps} x^{\eps}$ for 
all $\eps>0$.

Then there is a constant $1 \leq C \ll_{m_G,\ell} 1$ such that the following 
holds.
Let $E>C$ and suppose that for every $w(N)$-smooth integer $\tilde q \leq S(N)$ 
the sequence $(g(\tilde qn)\Gamma)_{n \leq N/\tilde q}$ is totally 
$\delta(N)^E$-equidistributed.
Let $L(m)=am+b$ be a linear polynomial with $0\leq b < a$ and a $w(N)$-smooth 
leading constant $a$, and let $q$ be a $w(N)$-smooth integer such that 
$qa \leq S(N)^{1/\ell^{\ell + 1}}$.
Then the finite sequence
$$
(g(aqm +b))\Gamma)_{m \leq T/(aq)}
$$
is totally $\delta(N)^{E/C}$-equidistributed.
\end{lemma}
\begin{proof}
Assuming that all conditions of \cite[Proposition 15.4]{lmr} are 
satisfied, this result, applied with $P=L$, will guarantee the existence of 
a $w(N)$-smooth integer $\tilde q$ and a constant $C= O_{m_G,\ell}(1)$ 
such that for every 
$0\leq r < \tilde q$ the sequence
$$
(g(aq (\tilde q m + r) +b))\Gamma)_{m \leq T/(aq\tilde q)}
$$
is totally $\delta(N)^{E/C}$-equidistributed.
This, however, implies that the sequence 
$$
(g(aqm +b))\Gamma)_{m \leq T/(aq)}
$$
itself is totally $\delta(N)^{E/C}$-equidistributed,
which will prove the lemma.

It remains to check that all conditions are satisfied.
The integer $\tilde q$ produced by \cite[Proposition 15.4]{lmr} comes from an 
application of \cite[Proposition 15.2]{lmr}.
The proof of the latter proposition reveals that we can take 
$\tilde q = (aq)^{C'}$ for some positive integer $C' = O_{\ell}(1)$. 
It is moreover possible to read of an explicit upper bound of the form 
$C' \leq \ell^{\ell + 1}$; 
c.f. the lines before \cite[equation (15.4)]{lmr} where $t$ is introduced and 
note that $C'$ corresponds to the quantity $td$.
In order for the proof of \cite[Proposition 15.4]{lmr} including its application 
of \cite[Proposition 15.2]{lmr} to work in the setting of the current lemma, it 
suffices to know that for every $w(N)$-smooth integers $q' \leq (aq)^{C'}$ the 
sequence $(g(q'm)\Gamma)_{m \leq T/q'}$ is totally 
$\delta(N)^{E}$-equidistributed. 
This, however, is guaranteed by the assumption that
$a q \leq S(N)^{1/\ell^{\ell + 1}}$.
\end{proof}

The following lemma will be used to carry out an inclusion-exclusion argument 
that allows us to reduce estimates involving $R^*_S$ to estimates involving $R$.
\begin{lemma}\label{l:d-scaling}
Let $N$ and $T$ be positive integers and suppose that $T=T(N)$ satisfies
$N^{1 - \eps} \ll_{\eps} T \leq N$ for all $\eps > 0$.
Suppose that $\delta: \NN \to (0,1/2)$ satisfies 
$$
\delta(x)^{-1} \asymp (\log w(x))^{C}
$$
for some positive constant $C$.
Let $m, \ell$ and $t$ be positive integers and suppose that $G/\Gamma$ is an 
$m$-dimensional nilmanifold equipped with a $\frac{1}{\delta(N)}$-rational 
Mal'cev basis $\mathcal{X}$ that is adapted to some filtration $G_{\bullet}$ of 
degree $\ell$.
Further, let $g \in \mathrm{poly}(\ZZ^t, G_{\bullet})$ be a polynomial sequence.
Given any integer $d \geq 1$, let $\x_d \in \{0, \dots, d^2\}^t$ be a fixed 
vector. 

Then there are constants $C_0 > 2t$ and $E_0>1$, both of order 
$O_{m,\ell, t}(1)$, such that, provided $N \gg_{m,\ell, t} 1$ is 
sufficiently large, the following holds for every $E>E_0$:

Suppose that $(g(\bn)\Gamma)_{\bn \in [T^{1/t}]^t}$ is totally 
$\delta(N)^E$-equidistributed.
Then, for every integer $K$ such that $1< K <T^{1/C_0}$ all but
$o(\delta(N) \frac{K}{\log w(N)})$
of the sequences
$$(g(d^2\bn + \x_d)\Gamma)_{\bn \in [T^{1/t}d^{-2}]^t }$$
for $d \in \{n \in [K,2K): \gcd(n,W(N))=1\}$ are totally 
$\delta(N)^{E/E_0}$-equidistributed. 
\end{lemma}

\begin{proof}
Let $W=W(N)$ and recall that
$\prod_{p<w(N)}(1-p^{-1}) \asymp \frac{1}{\log w(N)}$.
Let us abbreviate
$$
g_d(\x) = g(d^2\x + \x_d).
$$
Let $B>1$ and suppose $B=E/E_0$, with $E_0$ to be defined at the end of the proof.
Suppose further that there is some $K$, $1 < K < N^{1/C_0}$ such that  
$\gg \delta(N)\frac{K}{\log w(N)}$ of the integers $d \in [K,2K)$ with 
$\gcd(d,W(N))=1$ are exceptional.
In each of these cases, Lemma \ref{l:totaleq/eq} and Theorem \ref{t:GT-eq} 
imply that there is a non-trivial horizontal character 
$\eta_d:G/\Gamma \to \CC$ of modulus 
$|\eta_d| \ll \delta(N)^{-O_{m,\ell,t}(B)}$ such that
\begin{equation}\label{eq:smooth-d}
\|\eta_d \circ g_d\|_{C_*^{\infty}[T^{1/t}d^{-2}]^t} \ll 
\delta(N)^{-O_{m,\ell,t}(B)}. 
\end{equation}
By the pigeonhole principle, we find some $\eta$ such that $\eta_d = \eta$ for 
at least
$$\gg \delta(N)^{O_{m,\ell,t}(B)} \frac{K}{\log w(N)} $$ 
of the exceptional values of $d$.
Suppose $\eta \circ g$ has the representation
$$
\eta \circ g (\x) =
\sum_{\substack{i_1, \dots, i_t \geq 0 \\ i_1 + \dots + i_t \leq \ell}}
\beta_{i_1, \dots, i_t} x_1^{i_1} \dots x_t^{i_t}.
$$
Writing
$$
\eta \circ g_d (\x) =
\eta \circ g (d^2\x+ \x_d) =
\sum_{\substack{i_1, \dots, i_t \geq 0 \\ i_1 + \dots + i_t \leq \ell}}
\alpha^{(d)}_{i_1, \dots, i_t} x_1^{i_1} \dots x_t^{i_t},
$$
the bound \eqref{eq:smooth-d} translates to
\begin{equation}\label{eq:3.7-sup}
\sup_{(i_1 \dots, i_t) \not= \0} 
\left(\frac{T^{1/t}}{d^{2}}\right)^{i_1 + \dots + i_t}
\| \alpha^{(d)}_{i_1, \dots, i_t} \|
\ll \delta^{-O_{m,\ell,t}(B)}.
\end{equation}
Note that every coefficient $\alpha^{(d)}_{i_1, \dots, i_t}$ can be expressed in 
terms of coefficients $\beta_{j_1, \dots, j_t}$ as follows:
$$
\alpha^{(d)}_{i_1, \dots, i_t} = 
d^{2(i_1 + \dots + i_t)} \beta_{i_1, \dots, i_t} + 
\sum_{j_1 + \dots + j_t > i_1 +\dots + i_t} 
c_{\mathbf{i}, \mathbf{j}} 
\beta_{j_1, \dots, j_t},
$$
with $\mathbf{i}= (i_1, \dots,i_t)$, $\mathbf{j}= (j_1, \dots,j_t)$ and integral 
coefficients $c_{\mathbf{i}, \mathbf{j}}$ 
of order $O(d^{2(i_1 + \dots + i_t)})$.
Similarly as in the proof of Lemma \ref{l:poly-subsequences}, these identities 
allow us to deduce downward-inductively information on the coefficients 
$\beta_{j_1, \dots, j_t}$ 
from \eqref{eq:3.7-sup}: 

If $i_1 + \dots + i_t = \ell$, then we immediately have
$$
\|\beta_{i_1, \dots, i_t} d^{2 \ell} \| 
\ll \delta(N)^{-O_{m,\ell,t}(B)} d^{2\ell}T^{-\ell/t} .
$$
In general, we obtain
$$
\|
 \beta_{i_1, \dots, i_t} 
 d^{2 (\ell + (\ell - 1) + \dots + (i_1 + \dots + i_t))} 
\| 
\ll \delta(N)^{-O_{m,\ell,t}(B)}  
d^{2 (\ell + (\ell - 1) + \dots + (i_1 + \dots + i_t ))}
T^{-(i_1 + \dots + i_t)/t}.
$$
Thus,
$$
\|\beta_{i_1, \dots, i_t} d^{k} \| 
\ll \delta(N)^{-O_{m,\ell,t}(B)} d^k T^{-(i_1 + \dots + i_t)/t},
$$
for $k= \ell^2$.
The above bound holds for 
$\gg \delta(N)^{O_{m,\ell,t}(B)} \frac{K}{\log w(N)}
\gg \delta(N)^{O_{m,\ell,t}(B)} K$ 
values of $d \in [K,2K)$.
Employing the Waring-type result given in \cite[Lemma 3.3]{GT-nilmobius}, we 
deduce that there are at least
$\gg_{\ell} \delta(N)^{O_{\ell,t,m}(B)} K^{k}$
positive integers $n \leq 10^k K^k$ such that
$$
\|\beta_{i_1, \dots, i_t} n \| 
\ll \delta(N)^{-O_{m,\ell,t}(B)} K^k T^{-(i_1 + \dots + i_t)/t}.
$$
Setting $C_0 = 4tk$, so that $K^k \ll T^{1/2t}$, we deduce from the strong 
recurrence lemma recorded in \cite[Lemma 3.4]{GT-nilmobius} that 
for each $\beta_{i_1, \dots, i_t}$ there is a non-zero integer
$$
q_{i_1, \dots, i_t} \ll 
\delta(N)^{-O_{ m,\ell,t}(B)}
$$ 
such that 
$$\|q_{i_1, \dots, i_t} \beta_{i_1, \dots, i_t} \| 
\ll \delta(N)^{-O_{m,\ell,t}(B)} 
T^{-(i_1 + \dots + i_t)/t}.
$$
Let $q$ be the least common multiple of all the $q_{i_1, \dots, i_t}$.
Then $q \eta$ is a non-trivial horizontal character of modulus 
$|q \eta| \ll \delta(N)^{-O_{m, \ell,t}(B)}$
with the property that 
$$
\|q \eta \circ g\|_{C_*^{\infty}[T^{1/t}]^t} 
\ll \delta(N)^{-O_{m, \ell,t}(B)}.
$$
Recall that $B=E/E_0$. 
Choosing $E_0$ sufficiently large with respect to $m, \ell$ and $t$, we deduce 
from Theorem \ref{t:GT-eq} and Lemma \ref{l:totaleq/eq} that 
$(g(\bn)\Gamma)_{\bn \in [T^{1/t}]^t}$ fails to be totally 
$\delta(N)^{E}$-equidistributed, which is a contradiction.
\end{proof}

The final lemma of this section states the following. 
Given a multiparameter sequence, then we obtain a natural collection of 
one-parameter sequences by fixing all but one of the parameters. 
The lemma shows that if the multiparameter sequence is equidistributed then so 
are almost all of the one-parameter sequences from this collection.

\begin{lemma}\label{l:multi-to-one}
Let $m,t,\ell, N$ and $T$ be positive integers. 
Suppose that $N^{1 - \eps} \ll_{\eps} T \leq N$ for all $\eps > 0$, and
let $\delta: \NN \to (0,1/2)$ be 
such that $\delta(x)^{-1}\ll_{\eps} x^{\eps}$ for all $\eps>0$.
Let $G/\Gamma$ be an $m$-dimensional nilmanifold together with a 
$\frac{1}{\delta(N)}$-rational Mal'cev basis adapted to some filtration 
$G_{\bullet}$ of degree $\ell$.
Suppose that $g\in \mathrm{poly}(\ZZ^t, G_{\bullet})$. 
Any fixed choice of integers $a_1, \dots, a_{t-1}$ 
gives rise to an element $g_{a_1, \dots, a_{t-1}}$ of 
$\mathrm{poly}(\ZZ, G_{\bullet})$ by setting 
$g_{a_1, \dots, a_{t-1}}(n)= g(a_1 \dots, a_{t-1}, n)$.
Then there is a constant $1 \leq C \ll_{m,t,\ell} 1$ such that the 
following holds for all $E>C$, provided $T$ is sufficiently large.

If $(g(\bn)\Gamma)_{n \in [T^{1/t}]^t}$ is 
$\delta(N)^E$-equidistributed, then
$$(g_{a_1,\dots,a_{t-1}}(n)\Gamma)_{n \leq T^{1/t}}$$
is totally $\delta^{E/C}(N)$-equidistributed for all but 
$o(\delta(N)^{O_{m,t,\ell}(E/C)} T^{\frac{t-1}{t}})$ choices of 
$$1 \leq a_1, \dots, a_{t-1} \leq T^{1/t}.$$ 
\end{lemma}

\begin{proof}
Let $B>1$ denote the ratio $B=E/C$, with $C$ to be determined at the end of the 
proof. 
Suppose there are
$\gg \delta(N)^{O_{m,t,\ell}(B)} T^{(t-1)/t}$ tuples 
$(a_1, \dots, a_{t-1}) \in [1,T^{1/t}]^{t-1}$
for which $(g_{a_1,\dots,a_{t-1}}(n)\Gamma)_{n \leq T^{1/t}}$ fails to be totally
$\delta(N)^{B}$-equidistributed.

Applying Lemma~\ref{l:totaleq/eq} and Theorem~\ref{t:GT-eq}, we find non-trivial 
horizontal characters $\eta_{a_1, \dots, a_{t-1}}$ of modulus 
$\ll \delta(N)^{-O_{m,t,\ell}(B)}$ such that
\begin{equation}\label{eq:s-norm-again}
 \|\eta_{a_1, \dots, a_{t-1}} \circ g_{a_1, \dots, a_{t-1}} 
\|_{C^{\infty}_{*}[T^{1/t}]} \ll \delta(N)^{-O_{m,\ell,t}(B)}.
\end{equation}
By the pigeonhole principle there is some character $\eta$ such that
$\eta = \eta_{a_1, \dots, a_{t-1}}$ for at least 
$\gg \delta(N)^{O_{m,\ell,t}(B)} T^{(t-1)/t}$
of the exceptional tuples $(a_1, \dots, a_{t-1})$.
We continue to only consider this subset of exceptional $(t-1)$-tuples.

Suppose
$$
\eta \circ g(\bn) 
= \sum_{i_1, \dots, i_t} \gamma_{i_1, \dots, i_t} n_1^{i_1} \dots n_t^{i_t}.
$$
Then 
$$
\eta \circ g_{a_1,\dots,a_{t-1}}(n) 
= 
\sum_{j=1}^{\ell} n^j
\sum_{\substack{i_1+ \dots+ i_t \\ \leq \ell t -j}} 
a_1^{i_1} \dots a_{t-1}^{i_{t-1}}
\gamma_{i_1, \dots, i_{t-1}, j}
$$
and for any of the exceptional tuples from above, \eqref{eq:s-norm-again} 
translates to
\begin{equation}\label{eq:s-norm'}
\sup_{1 \leq j \leq \ell} T^{j/t} 
\Big\|
 \sum_{\substack{i_1+ \dots+ i_t \\ \leq \ell t -j}} 
 a_1^{i_1} \dots a_{t-1}^{i_{t-1}}
 \gamma_{i_1, \dots, i_{t-1}, j}
\Big\| \ll \delta(N)^{-O_{m,t,\ell}(B)}. 
\end{equation}
Each of the coefficients of $\eta \circ g_{a_1,\dots,a_{t-1}}(n)$ can be 
regarded as a polynomial in $t-1$ variables that is evaluated at the point
$(a_1, \dots, a_{t-1})$.
These polynomials take the form
$$P_{j} (n_1, \dots, n_{t-1})
=  
\sum_{\substack{i_1+ \dots+ i_t \\ \leq \ell t -j}} 
n_1^{i_1} \dots n_{t-1}^{i_{t-1}}
\gamma_{i_1, \dots, i_{t-1}, j} .
$$
The bounds \eqref{eq:s-norm'} show that to each of these polynomials 
\cite[Proposition 2.2]{GT-erratum} applies with 
$\eps=\delta(N)^{-O_{m,t,\ell}(B)} T^{-j/t}$.
Thus there is $Q_{j} \ll \delta(N)^{-O_{m,t,\ell}(B)}$ such that
$$
\|Q_j P_{j} \|_{C_*^{\infty}[T^{1/t}]^{t-1}}
=
\sup_{1\leq i_1+ \dots+ i_{t-1} \leq \ell t -j} 
T^{(i_1 + \dots + i_{t-1})/t} 
\| Q_{j} \gamma_{i_1, \dots, i_{t-1}, j}\| 
\ll \delta(N)^{-O_{m,t,\ell}(B)} T^{-j/t},
$$
and hence
$$
\sup_{1\leq i_1+ \dots+ i_{t-1} \leq \ell t -j} 
T^{(i_1 + \dots + i_{t-1} + j)/t}  
\| Q_j \gamma_{i_1, \dots, i_{t-1},j}\| \ll \delta(N)^{-O_{m,t,\ell}(B)}.
$$
Since $\delta(x)^{-1} \ll_{\eps} x^{\eps}$ and hence
$\| Q_j \gamma_{i_1, \dots, i_{t-1},j}\| = o(1)$,
we can introduce a factor 
$Q_{1} \dots Q_{\ell}/Q_j$ into the latter expression and deduce 
that
$$
\sup_{1\leq i_1+ \dots+ i_t \leq \ell t} 
T^{(i_1 + \dots + i_{t})/t} 
\| Q_{1} \dots Q_{\ell} \gamma_{i_1, \dots, i_t}\| \ll 
\delta(N)^{-O_{m,t,\ell}(B)},
$$
provided $T$ and $N$ are sufficiently large.
For $B=E/C$ with $C>1$ sufficiently large depending only on $m,t$ and $\ell$, the 
latter bound implies in view of Theorem \ref{t:GT-eq} that 
$(g(\bn)\Gamma)_{\bn \in [T^{1/t}]^t}$ fails to be 
$\delta(N)^{E}$-equidistributed. 
This contradicts our assumptions and completes the proof.
\end{proof}

\section{Proof of Theorem \ref{p:nilsequences}} \label{s:proof}
The general strategy for proving results like Theorem \ref{p:nilsequences} is 
to deduce them from the special case in which the nilsequence involved is 
equidistributed.
This strategy was introduced in \cite[\S2]{GT-nilmobius}.
The transition from the equidistributed statement to the general one is 
achieved through an application of the factorisation theorem for nilsequences, 
c.f.\ \cite[Theorem 1.19]{GT-polyorbits}, or a consequence thereof.
For technical reasons we require a factorisation result of the form given in
\cite[Theorem 3]{lmf}, which is a slight generalisation of
\cite[Proposition 16.4]{lmr} and arises by iterative application of 
\cite[Theorem 1.19]{GT-polyorbits}.
Combining this factorisation theorem with the major arc estimate given in 
Lemma \ref{l:major-arc}, we will deduce Theorem \ref{p:nilsequences} from the 
following adaptation of \cite[Proposition 6.4]{bm} to the square-free version of 
$R$.   

\begin{proposition}\label{p:nils-equi}
Let $N$ and $T=T(N)$ be positive integers such that
$N^{1-\eps} \ll_{\eps} T \leq N$ for all $\eps > 0$.
Let $\epsilon\in \{\pm\}$, and let $W=W(N)$. 
Let $S$ be a finite set of primes, all bounded by $w(N)$, and let $A$ be an 
integer such that $A \Mod{W} \in \cA(R^*_S,N)$ and $0 \leq \epsilon A < W$.
Suppose that $\delta:\ZZ \to (0,1/2)$ satisfies 
$\delta^{-1}(x) \asymp (\log w(x))^{C}$ for some positive constant $C$.
Assume that $(G/\Gamma, d_{\mathcal X})$ is an $m_G$-dimensional
nilmanifold with a filtration $G_{\bullet}$ of degree $\ell$ adapted to a 
$\frac{1}{\delta(N)}$-rational Mal'cev basis. 
Let $g \in \mathrm{poly}(\ZZ,G_{\bullet})$. 
Let $S:\NN \to \NN$ be such that $S(x) \ll_{\eps} x^{\eps}$ for all $\eps >0$.
Finally, let $E>0$ and suppose that for every $w(N)$-smooth number 
$\tilde q \leq S(N)$ the finite sequence 
$(g(\tilde q m)\Gamma)_{0<m\leq T/\tilde q}$ is totally 
$\delta(N)^E$-equidistributed in $G/\Gamma$. 

Then the following holds.
There exists a constant $E_0>1$, only depending on $m_G,\ell$ and $n:=[K:\QQ]$ 
such that for every Lipschitz function $F:G/\Gamma \to [-1,1]$ of mean value
$\int_{G/\Gamma}F=0$, 
for every $w(N)$-smooth number $q>0$ such that $(Wq)^{n\ell^2} \leq S(N)$, and
for every $0\leq b < q$ we have
\begin{align*}
\bigg|\frac{1}{T}\sum_{0< \epsilon m \leq T} &R^*(W (qm+b) + A) 
F(g(|m|)\Gamma)\bigg| \\
&\ll_{m_G, \ell, n} 
\delta(N)^{E/E_0} (1+\|F\|_{\mathrm{Lip}}) 
\frac{\rho(W,A)}{W^{n-1}}
\prod_{p>w(N)} 
\bigg( 1 - \frac{\rho(p^2,0)}{p^2}\bigg)
\end{align*}
as $N \to \infty$.
\end{proposition}

\begin{proof}[Proof of Theorem \ref{p:nilsequences} assuming Proposition 
\ref{p:nils-equi}]
We follow the strategy from \cite[\S2]{GT-nilmobius}.
To prove Theorem \ref{p:nilsequences} we may restrict attention to the case 
where $Q \leq \log \log \log N$ as the statement is trivially true otherwise.
This allows us to apply \cite[Theorem 3]{lmf} with the following parameters 
(distinguished by a tilde from already existing ones):
$$
\tilde N=N, \quad
\tilde T= T/W(N), \quad
\tilde k= w(N)=\log \log N,\quad
\tilde Q_0= \log \tilde k, \quad 
\tilde R=W(N)
$$
and, finally, $\tilde B = E$ and 
$\tilde E > 2n \ell^2 \ell^{\ell + 1}$.
Observe that $W(N) \ll (\log N)^{C_1 \log \log N}$. 
Whence, $\tilde R$ satisfies the required condition that
$\tilde R(N)^t \ll_t N$ for all $t>0$ as $N \to \infty$.
By \cite[Theorem 3]{lmf} we therefore obtain an integer $Q'$ such that
$$\log \log \log N \leq Q' \ll (\log \log \log N)^{O_{m_G, \ell,E}(1)}$$
and a partition $\mathcal{P}$ of the set 
$\{1, \dots, \tilde T\}$ into $\ll W^{O_{\ell, m_G, n, E}(1)}$ 
disjoint subprogressions, each of some $w(N)$-common difference 
$q(P) \ll W^{O_{\ell, m_G, n, E}(1)}$ and length $\frac{T}{Wq(P)} + O(1)$.
Along each progression $P = \{a < n < b : n \equiv r \Mod{q}\}$ from 
$\mathcal{P}$, the polynomial sequence $g$ has a factorisation
$$g(qm+r) = \eps_P(m) g'_P(m)\gamma_P(m)$$
with the properties that
\begin{enumerate}
 \item $\eps_P: \ZZ \to G$ is $(Q',\frac{T}{Wq})$-smooth\footnote{The notion of 
 smoothness was defined in \cite[Def.\ 1.18]{GT-polyorbits}. 
 A sequence $(\eps (n))_{n\in \ZZ}$ is said to be \emph{$(M,N)$-smooth} if 
 both $d_{\mathcal{X}}(\eps(n),\id_G) \leq N$ and 
 $d_{\mathcal{X}}(\eps(n),\eps(n-1)) \leq M/N$ hold 
 for all $1 \leq n \leq N$.};
 \item
 $g'_P: \ZZ \to G'$ takes values in $G'$ and for each $w(N)$-smooth number 
$\tilde q < (q q_{\gamma_P}W)^{\tilde E}$ the finite sequence 
$(g'_P(\tilde q m)\Gamma')_{m \leq T/(Wq \tilde q)}$ is totally
${Q'}^{-E}$-equidistributed in $G'\Gamma/\Gamma$;
 \item $\gamma_P: \ZZ \to G$ is a product 
 $\gamma_P(m) = \gamma_1(m) \dots \gamma_{t}(m)$ of length at most $m_G$ of
 $Q$-rational sequences $\gamma_i$. 
 It gives rise to a periodic sequence $(\gamma_P(m)\Gamma)_{m \in \ZZ}$ with  
 $w(N)$-smooth period $q_{\gamma_P} \leq Q'$.
\end{enumerate}
Any progression $P \in \mathcal{P}^*$ as above can be split into 
$q_{\gamma_P} \leq Q'$ 
subprogressions on which $\gamma_P$ is constant.
Next, we split each of these subprogressions into pieces of diameter between
$Q'^{-c^*E}T/W$ and $2Q'^{-c^*E}T/W$ for a small parameter $c^* \in (0,1)$ which 
will be determined later.
Let $\mathcal{P}^*$ denote the collection of the resulting bounded diameter 
pieces of all progressions $P\in \mathcal{P}$ 
and note that each $P' \in \mathcal{P}^*$ is of the form
\begin{equation}\label{e:P'}
P' = \left\{q q_{\gamma} m + r' : 
\frac{\delta_1 T}{Wqq_{\gamma}} < m \leq \frac{\delta_2 T}{Wqq_{\gamma}} \right\},
\end{equation}
for a $w(T)$-smooth period $q$ as before, for some $0 \leq r' < q q_{\gamma}$, 
and for
$\delta_1, \delta_2 \in [0,1]$ such that
$Q'^{-c^*E} \leq \delta_2 - \delta_1 \leq 2Q'^{-c^*E}$ and either
$\delta_1=0$ or $\delta_1 > Q'^{-c^*E}$.
For each progression $P' \in \mathcal{P}^*$ let $s_{P'}$ denote its smallest 
element.
If $F: G/\Gamma \to \CC$ is a Lipschitz function, then the right-invariance of 
the metric $d_{\mathcal{X}}$, defined in \cite[Definition 2.2]{GT-polyorbits}, 
implies for any $m,m'$ that belong to the same element of $\mathcal{P}^*$ that
\begin{align*}
&|F(\eps(m) g'(m) \gamma(m) \Gamma)
- F(\eps(m') g'(m) \gamma(m) \Gamma)|\\
&\leq \|F\|_{\mathrm{Lip}}~ 
d_{\mathcal{X}}(\eps(m) g'(m) \gamma(m),\eps(m') g'(m) \gamma(m)) \\
&= \|F\|_{\mathrm{Lip}}~ d_{\mathcal{X}}(\eps(m),\eps(m') ) \\
&\leq \|F\|_{\mathrm{Lip}}~ |m-m'|Q'W/T \\
&\ll \|F\|_{\mathrm{Lip}}~ Q'^{1-c^*E}
\end{align*}
if $\eps$, $\gamma$ and $g'$ satisfy the respective conditions (1)--(3) from 
above.
This estimate allows one to fix for any $P' \in \mathcal{P}^*$ the contribution 
of $\eps$.
To see this, suppose $P'$ arises from the progression $P\in \mathcal{P}$, and 
define
$$
\mu_{R^*} :=
\kappa^{\epsilon}
    \frac{\rho(W,A)}{W^{n-1}}
    \prod_{p>w(N)} 
    \bigg( 1 -
    \frac{\rho(p^2,0)}{p^{2n}} \bigg).
$$
By combining the previous estimate with Lemmas \ref{l:R-mean-value} and
\ref{l:major-arc}, we deduce that
\begin{align*}
&\sum_{m\in P'} \Big(R^*(W m + A) - \mu_{R^*}\Big) F(g(m)\Gamma) \\
&= \sum_{m\in P'} \Big(R^*(W m + A) - \mu_{R^*}\Big) 
  F(\eps_P(s_{P'}) g'_P(m) \gamma_P(m) \Gamma) \\
&\quad + O\left(\|F\|_{\mathrm{Lip}} Q'^{1-c^*E}
  \left(\mu_{R^*} \#P' + W^{O_{l,m_G,n,E}(1)} T^{1 - \frac{1}{20n}}\right)
  \right).
\end{align*}
In view of Theorem \ref{p:nilsequences} the error term above is still acceptable 
when summed over all $P' \in \mathcal{P}*$, which allows us to exclude it from 
further observation.

The remaining sequence $m \mapsto F(\eps(s_{P'}) g'(m) \gamma(s_{P'}) \Gamma)$ 
can be reinterpreted as an equidistributed nilsequence, as is shown in 
\cite[\emph{Claim} in \S2]{GT-nilmobius}.
Indeed, setting
\begin{align*}
 H_{P'} &:= \gamma_{P}(s_{P'})^{-1} G' \gamma_{P}(s_{P'}),\\
(H_{P'})_{\bullet} &:= \gamma_{P}(s_{P'})^{-1} G'_{\bullet} \gamma_{P}(s_{P'}),\\
\Lambda_{P'} &:= \Gamma \cap H_{P'},\\
g_{P'}(m) &:= \gamma_{P}(s_{P'})^{-1} g'_{P}(m) \gamma_{P}(s_{P'})
\end{align*}
and
$$
F_{P'}: H_{P'}/\Lambda_{P'} \to [-1,1], \qquad 
F_{P'}(x\Lambda_{P'}) := F(\eps_{P}(s_{P'}) \gamma_{P}(s_{P'}) x \Gamma),
$$
the \emph{Claim} guarantees the existence of a Mal'cev basis for 
$H_{P'}/\Lambda_{P'}$ 
adapted to the filtration $(H_{P'})_{\bullet}$ that is ${Q'}^{O(1)}$-rational in 
terms of the basis $\mathcal{X}$.
This basis induces a metric on $H_{P'}/\Lambda_{P'}$ with respect to which we 
have on the one hand $\|F_{P'}\|_{\mathrm{Lip}} \leq {Q'}^{O(1)}$ and on the 
other hand each of the sequences 
$(g_{P'}(\tilde{q}m))_{m \leq T/(Wq(P)\tilde{q})}$ for 
$w(N)$-smooth $\tilde{q} < (q(P) q_{\gamma_P}W)^{2n\ell^2 \ell^{\ell+1}}$ is 
totally ${Q'}^{-cE+O(1)}$-equidistributed for some 
$c>0$ depending only on $m_G$ and $d$.

We aim to apply Proposition \ref{p:nils-equi} making use of the 
equidistribution properties of $g_{P'}$.
To prepare this application, first note that Lemma \ref{l:weaker-assumptions} 
implies that every sequence 
$$g_{P',r}(m) := g_{P'}(q_{\gamma_P} m + r), \quad (0 \leq r < q_{\gamma_P}),$$
has the property that for every
$w(N)$-smooth integer  
$\tilde{q} < (q(P)q_{\gamma_P}W)^{2n\ell^2}/q_{\gamma_P}$
the sequence
$(g_{P',r}(\tilde{q}m))_{m \leq T/(Wq(P)q_{\gamma_P}\tilde{q})}$ 
is totally ${Q'}^{-c'E+O(1)}$-equidistributed for some 
$c'>0$ depending only on $m_G$ and $d$.
Thus, we may set $S_{P'}(N):=(q(P)q_{\gamma_P}W)^{n\ell^2}$ for any 
$P' \subset P \in \mathcal{P}$.  
This quantity will play the role of $S(N)$ from Proposition \ref{p:nils-equi}.
It will, in particular, allow us to take $q=q(P)q_{\gamma_P}$ in the 
application below.

Next, we need to ensure that the mean value of the Lipschitz function Proposition 
\ref{p:nils-equi} will be applied with vanishes.
To this regard, note that Lemma \ref{l:R-mean-value} implies
$$ 
\sum_{m\in P'} \Big(R^*(W m + A) - \mu_{R^*}\Big) 
\ll q(P)q_{\gamma} W (T/W)^{1 - \frac{1}{20n}}
\ll_{\eps} T^{1 - \frac{1}{20n} + \eps}.
$$
This allows us to replace $F_{P'}$ by
$(F_{P'} - \int_{H_{P'}/\Lambda_{P'}} F_{P'})$ since
\begin{align*}
&\sum_{m\in P'} \Big(R^*(W m + A) - \mu_{R^*}\Big) 
  F(\eps_P(m_{P'}) g'_P(m) \gamma_P(m) \Gamma) \\
&= \sum_{\frac{\delta_1T}{Wqq_{\gamma}} < m \leq 
\frac{\delta_2T}{Wqq_{\gamma}}} 
 \Big(R^*(W (q(P)q_{\gamma} m + r') + A) - \mu_{R^*}\Big) 
  \left(F_{P'}(g_{P',r}(m)\Lambda_{P'})
  - \int_{H_{P'}/\Lambda_{P'}} F_{P'}\right)\\
&\qquad \qquad \qquad
  + O_{\eps} \left(T^{1 - \frac{1}{20n} + \eps} \right).
\end{align*}
Here we made use of the notation from \eqref{e:P'}.
Note that the error term is negligible even when summed over all 
$P' \in \mathcal{P^*}$ and can be ignored.

We are now ready to apply Proposition \ref{p:nils-equi} for each 
$P'\in \mathcal{P}^*$ to the sum on the left hand side above with 
$\delta(N)= (\log \log \log N)^{-1}$ and with  
$\delta(N)^E$ replaced by ${Q'}^{-c'E+O(1)}$. 
Since Proposition \ref{p:nils-equi} cannot be applied to the short progression 
$P'$ directly, we apply it once with $T$ replaced by $\delta_2 T/(Wqq_{\gamma})$ 
and once with $T$ replaced by $\delta_1 T/(Wqq_{\gamma})$.
Resulting from this double-counting, we pick up a factor of $O(Q'^{c^*E})$ 
along each progression $P$ from the original decomposition~$\mathcal{P}$.
In total we obtain
\begin{align*}
& \sum_{P \in \mathcal{P}}
  \sum_{\substack{P' \in \mathcal{P}^*\\ P' \subset P}}
  \sum_{n\in P'} \Big(R^*(W m + A) - \mu_{R^*}\Big) 
  \left(F_{P'}(g'(m)\Lambda_{P'})
  - \int_{H_{P'}/\Lambda_{P'}} F_{P'}\right)\\
&\ll_{m_G, \ell, n} 
\left(\sum_{P \in \mathcal{P}} Q'^{c^*E} \# P\right)
{Q'}^{- cE/E_0 + O(1)} 
(1+\|F\|_{\mathrm{Lip}}) 
\frac{\rho(W,A)}{W^{n-1}}
\prod_{p>w(T)} 
\bigg( 1 - \frac{\rho(p^2,0)}{p^2}\bigg)\\
&\ll_{m_G, \ell, n} 
\frac{T}{W}
{Q'}^{c^*E - cE/E_0 + O(1)} 
(1+\|F\|_{\mathrm{Lip}}) 
\frac{\rho(W,A)}{W^{n-1}}
\prod_{p>w(T)} 
\bigg( 1 - \frac{\rho(p^2,0)}{p^2}\bigg).
\end{align*}
Choosing $c^* < c/(2E_0)$ and recalling that $Q' \geq \log \log \log N$, 
this implies the result.
\end{proof}

\begin{proof}[Proof of Proposition \ref{p:nils-equi}]
We observe first of all that \eqref{eq:rho-p^2-0} implies that 
$$\prod_{p>w(N)} 
\bigg( 1 - \frac{\rho(p^2,0)}{p^2}\bigg)
\asymp 1,$$
provided $N$ is sufficiently large, so that the last factor in the bound can be 
ignored.

Our aim is to deduce this proposition from the proof of 
\cite[Proposition 6.4]{bm} by using the decomposition 
$ R^*_S(m) = \sum_{d^2 | m, (d,W)=1} \mu(d) R(m)$, valid for all $m$ with
$m \Mod{W} \in \mathcal{A}(R_S^*,N)$, together with Lemmas 
\ref{l:poly-subsequences}, \ref{l:d-scaling} and \ref{l:multi-to-one}.
Indeed, the decomposition yields
\begin{align} \label{eq:lattice}
\nonumber
&\frac{1}{T}
 \sum_{0< \epsilon m \leq T} R^*_S(W(qm + b)+A) F(g(|m|)\Gamma) \\
\nonumber
&=  \sum_{\substack{d \leq N^{1/2}\\ \gcd(d,W)=1}} 
\frac{\mu(d)}{T}
\sum_{\substack{0< \epsilon m \leq T \\
      W(qm + b)+A \equiv 0 \Mod{d^2}}}
R(W(qm + b)+A) F(g(|m|)\Gamma) \\
&=  \sum_{\substack{d \leq N^{1/2}\\ \gcd(d,W)=1}}
\frac{\mu(d)}{T}
\sum_{\substack{0< \epsilon m \leq B 
  \\ m\equiv A+Wb \Mod{Wq}
  \\ m\equiv 0 \Mod{d^2}}} R(m)
  F\Big(
    g\Big(\frac{m-A-  Wb} {\epsilon Wq}\Big) \Gamma
   \Big),
\end{align}
where $B := Wq(T + b) + A \asymp WqT$.
For fixed $d$, the inner sum will be estimated using the strategy from 
\cite[Proposition 6.4]{bm}.
Due to the additional restriction $m \equiv 0 \Mod{d^2}$ that appears here, we 
need to work with subsequences of the nilsequence that mattered in 
\cite[Proposition 6.4]{bm}.
The results from Section \ref{s:technical} provide the necessary information on 
equidistribution properties of the new sequences as $d$ varies.
Since all these results require $d$ to be sufficiently small, we begin by 
restricting the summation in $d$. 

The estimate \eqref{eq:tail-est} allows us to remove large values of $d$ from 
consideration.
In particular, it shows that the summation can be truncated at $d \leq N^{1/C_0}$ 
for any fixed $C_0>2$ at the expense of an error term of order 
$O(T^{-1}(WqT)^{1-1/O(C_0)})$.
Since $Wq \ll_{\eps} T^{\eps}$, this error is $o(1)$.
In order to obtain a direct sum over $F$, we aim to move the appearance of the 
$R$ function in \eqref{eq:lattice} from the argument of the summation to the 
summation condition.
For this purpose, let 
$$
\mathcal Y = 
 \left\{ 
\y \in (\ZZ/Wq\ZZ)^n: 
\begin{array}{l}
 \bN_K(\y) \equiv A+Wb \Mod{Wq} 
\end{array}
 \right\},
$$
so that $\#\mathcal Y = \rho(Wq, A+Wb)$. 
Given $\y \in \mathcal{Y}$ and $d$ with $\gcd(d,W)=1$, we further set
$$
\mathcal X_{\y,d} = 
 \Big\{ 
\tilde \x_d \in (\ZZ/d^2\ZZ)^n: 
\begin{array}{l}
 \bN_K(Wq \tilde\x_d + \y) \equiv 0 \Mod{d^2}
\end{array}
 \Big\},
$$
so that $\#\mathcal X_{\y,d} = \rho(d^2, 0)$.
Recall that
$$R(m)=
\1_{m \not= 0} \cdot~
\#\left\{\mathbf{x} \in \ZZ^{n}\cap\mathfrak{D}^+:
\begin{array}{l}
\bN_{K}(\x)=m
\end{array}
\right\}.
$$
Thus \eqref{eq:lattice} becomes 
\begin{align}\label{eq:nilsequ-1}
\nonumber
  \sum_{\substack{d \leq N^{1/C_0}\\ \gcd(d,W)=1}}
  \frac{\mu(d)}{T''}
& \sum_{\y \in \mathcal{Y}}
  \sum_{\tilde\x_d \in \mathcal{X}_{\y,d}}
  \sum_{\substack{ \x \in \ZZ^n \\
        Wq(d^2\x + \tilde \x_d)+ \y \\ \in 
  B^{1/n} \mathfrak{X}(1) }}
  F\Big(g\Big(
   \frac{\bN_K(Wq(d^2\x + \tilde \x_d)+ \y)-A - Wb}{\epsilon Wq}\Big)
   \Gamma\Big) \\
& + o(1),
\end{align}
where 
$\mathfrak{X}(1)
= \{\x\in \mathfrak{D}^+ : 0<\epsilon\nf_K(\x)\leq 1\}$ is a compact set.

As in \cite{bm}, we proceed with the analysis of the inner sum over $\x$ by 
fixing the first $n-1$ coordinates of $\x$.
Since $\mathfrak{X}(1)$ is compact, we have 
$\mathfrak{X}(1) \subset (-\alpha,\alpha)^n$ for some positive constant
$\alpha =O(1)$ which allows us to restrict the ranges of the first $n-1$ 
coordinates we consider:
Letting $\pi:\RR^n \to \RR^{n-1}$ denote the projection onto the
coordinate plane $\{x_n=0\}$, it thus suffices to consider the set 
$\{\x:\pi(\x) = \a\}$, where $\a$ runs over all points in 
$$\ZZ^{n-1}\cap \Big(\frac{B^{1/n}(-\alpha,\alpha)^{n-1} - Wq \pi(\tilde \x_d) 
-\y}{Wqd^2}\Big).$$
Assuming that $T$ is sufficiently large, the latter set can, in fact, be replaced 
by
\begin{equation} \label{eq:Z_d-def}
\mathcal{Z}_{d} =
\ZZ^{n-1}\cap \frac{B^{1/n}}{Wqd^2}(-2\alpha,2\alpha)^{n-1}. 
\end{equation}

Returning to \eqref{eq:nilsequ-1}, we consider the argument of $g$.
Since the coefficient of $X_n^n$ in $\bN_K(X_1, \dots, X_n)$ is non-zero,
we obtain an integral polynomial of degree $n$ and leading coefficient
$\epsilon N_{K/\QQ}(\omega_n) d^{2n}(Wq)^{n-1}$ when fixing all but the 
$n$th variable in
$$
\frac{\bN_K(Wq(d^2\x + \tilde \x_d)+ \y)-A-  Wb}{\epsilon Wq}.
$$ 
If $\pi(\x) =\a$, we let $P_{\a;\y;d}(x)$ denote this polynomial.

Our next step is to show that most of the sequences 
$g(P_{\a;\y;d}(x)\Gamma)_{|x| \ll B^{1/n} / Wqd^2}$ are equidistributed.
The assumptions on $g$ and Lemma \ref{l:poly-subsequences} imply that the 
sequence
$$
\Big(g\Big(\frac{\bN_K(Wq \x + \y)-A-  Wb}{\epsilon Wq}\Big)\Gamma\Big)_
{\x \in [B^{1/n}/(Wq)]^n}
$$ 
is totally $\delta(N)^{E/C'}$-equidistributed for some 
$1 \leq C' \ll_{m_G, \ell, n} 1$, provided $E>C'$.
Applying Lemma \ref{l:d-scaling} to this sequence, we thus deduce that there is 
$1 \leq C'' \ll_{m_G, \ell, n} 1$ such that the sequence
$$
\Big(g\Big(
   \frac{\bN_K(Wq(d^2\x + \tilde \x_d)+ \y)-A - Wb}{\epsilon Wq}\Big)
   \Gamma\Big)_{\x \in [B^{1/n}/(Wqd^2)]^n}
$$
is totally $\delta(N)^{E/C''}$-equidistributed for all but 
$o(\delta(N)^{O_{m_G, \ell, n}(E/C'')}K)$ integers $d$ such that
$d \in [K,2K)$ and $\gcd(d,W)=1$, and for all integers 
$K \in (1,N^{1/C_0})$, provided $E>C''$.
Finally, for all unexceptional values of $d$, Lemma \ref{l:multi-to-one} implies 
that there is $1 \leq C''' \ll_{m_G, \ell, n} 1$ such that, provided $E>C'''$,
the sequence
$$
\Big(g\Big(P_{(a_1,\dots,a_{n-1});\y;d}(x)\Big)
   \Gamma\Big)_{x \leq B^{1/n}/(Wqd^2)}
$$
is totally $\delta(N)^{E/C'''}$-equidistributed for all but 
$o(\delta(N)^{O_{m_G, \ell, n}(E/C''')} (\frac{B^{1/n}}{Wqd^2})^{n-1})$ 
integer tuples $(a_1, \dots, a_{n-1})$ with
$1 \leq a_1, \dots, a_{n-1} \leq B^{1/n}/(Wqd^2)$.

Before we exploit these equidistribution properties, we aim to bound the 
contribution of exceptional values of $d$ and $\a$.
Using the fact that $\|F\|_{\infty} \ll 1$, the contribution of exceptional 
values for $d$ to the main term of \eqref{eq:nilsequ-1} may trivially be bounded
by
\begin{align*}
& \frac{1}{T}
  \#\mathcal{Y}
  \sum_{k=0}^{\frac{\log N}{C_0 \log 2}}
  \sum_{\substack{d \sim 2^k \\ \gcd(d,W)=1 \\ d \text{ exceptional}}}
  \rho(d^2,0)
  \sum_{\a \in \mathcal{Z}_d}
  \frac{B^{1/n}}{Wq d^2}\\
&\ll \frac{1}{T}
  \rho(Wq,A+Wb)
  \sum_{k=0}^{\frac{\log N}{C_0 \log 2}}
  \sum_{\substack{d \sim 2^k \\ \gcd(d,W)=1 \\ d \text{ exceptional}}}
  \rho(d^2,0)
  \frac{TWq}{(Wq d^2)^n}. 
\end{align*}  
In view of \eqref{eq:lifting}, the above is further bounded by
\begin{align*}  
&\ll \frac{\rho(W,A)}{W^{n-1}} 
  \sum_{k=0}^{\frac{\log N}{C_0 \log 2}}
  \sum_{\substack{d \sim 2^k \\ \gcd(d,W)=1 \\ d \text{ exceptional}}}
  \frac{\rho(d^2,0)}{d^2} \\
&\ll \frac{\rho(W,A)}{W^{n-1}} 
  \sum_{k=0}^{\frac{\log N}{C_0 \log 2}} 
  \delta(N)^{O_{m_G, \ell, n}(E/C'')} (2^{1-\frac{1}{2}})^{-k} \\
& \ll \frac{\rho(W,A)}{W^{n-1}}
  \delta(N)^{O_{m_G, \ell, n}(E/C'')} ,
\end{align*}
where we employed \eqref{eq:rho(d^2,0)-bound} with $\eps=\frac{1}{2}$.

Recall the definition of $\mathcal{Z}_d$ from \eqref{eq:Z_d-def}.
Then the contribution from exceptional values of $\a$ may be bounded in 
a similar manner: 
\begin{align*}
&  \frac{1}{T}
  \#\mathcal{Y}
  \sum_{\substack{d \leq N^{1/C_0} \\ \gcd(d,W)=1}}
  \rho(d^2,0)
  \sum_{\substack{\a \in \mathcal{Z}_d\\ \a \text{ exceptional}}}
  \frac{B^{1/n}}{Wq d^2}\\
&\ll \frac{1}{T}   \rho(Wq,A+Wb)
  \sum_{\substack{d \leq N^{1/C_0} \\ \gcd(d,W)=1}}
  \rho(d^2,0) \delta(N)^{O_{m_G, \ell, n}(E/C''')}
  \frac{TqW}{(Wq d^2)^n} 
  \\
&\ll \delta(N)^{O_{m_G, \ell, n}(E/C''')} \frac{\rho(W,A)}{W^{n-1}} 
  \sum_{k=0}^{\frac{\log N}{C_0 \log 2}}
  \sum_{\substack{d \leq N^{1/C_0} \\ \gcd(d,W)=1}}
  \frac{\rho(d^2,0)}{d^2} \\
&\ll \delta(N)^{O_{m_G, \ell, n}(E/C''')} 
\frac{\rho(W,A)}{W^{n-1}} .
\end{align*}

In the case of unexceptional values of $d$ and $a_1, \dots, a_{n-1}$, we can 
finally proceed in exactly the same way as in the proof of 
\cite[Proposition 6.4]{bm}.
First of all, we recall from \cite{bm} how the lines 
$\{\x \in \RR: \pi(\x) = (a_1, \dots, a_{n-1})\}$ intersect the domain 
$\mathfrak{X}(1) \subset (-\alpha,\alpha)^n$:
Let $\a'=(a'_1,\dots,a'_{n-1})$, with  $|\mathbf{a}'| < \alpha$,
and consider the line 
$\ell_{\a'}: (-\alpha,\alpha) \to \RR^n$
given by $\ell_{\a'}(x) = (\a',x)$.
For $\eps \geq 0$, let $\partial_{\eps} \mathfrak{X}(1) \subset \RR^n$
denote the set of points at distance at most $\eps$ to the boundary of
the closure of $\mathfrak{X}(1)$.
We note that the set
$$\left\{x \in (-\alpha,\alpha): \ell_{\a'}(x) \in 
\mathfrak{X}(1) \setminus \partial_{0}\mathfrak{X}(1) \right\}$$
is the union of disjoint open intervals.
By removing all intervals of length at most $\eps$, we obtain a collection
of at most $2 \alpha \eps^{-1} \ll \eps^{-1}$ open intervals
$I_1(\a'), \dots, I_{k(\a')}(\a') \in (-\alpha, \alpha)$ such
that any $x\in (-\alpha, \alpha)$ satisfies the implication
\begin{align*}
\ell_{\a'}(x) \in 
\mathfrak{X}(1) \setminus \partial_{\eps}\mathfrak{X}(1)
\implies 
x \in I_j(\a') \text{ for some }j \in \{1,\dots,k(\a')\}.
\end{align*}
We will choose a suitable value of $\eps$ at the end of the proof.

Observe that any interval $(z_0,z_1) \subset (-\alpha,\alpha)$ can be
expressed as a difference of intervals in $(-\alpha,\alpha)$ that have
length at least $2\alpha/3$.
Indeed, $z_0$ and $z_1$ partition $(-\alpha,\alpha)$ into three (possibly
empty) intervals, at least one of which has length at least
$2\alpha/3$.
Thus, one of the three representations
$$
(z_0,z_1)
=(-\alpha, z_1) \setminus (-\alpha, z_0] 
=(z_0, \alpha)\setminus[z_1, \alpha)
$$
has the required property. 
For each $\a'$ and $j \in \{1,\dots,k(\a')\}$, we let 
$I_j(\a') = J_j^{(1)}(\a')\setminus J_j^{(2)}(\a')$ be such a
decomposition, where 
$J_j^{(2)}(\a')$
is possibly empty.

Given any $\a \in \ZZ^{n-1}$, we let $\a'$ denote from now on the specific vector
$$\a'=B^{-1/n}(Wq d^2\a + \pi(Wq \tilde\x_d + \y)).$$
With this notation, the main term of \eqref{eq:nilsequ-1} equals
\begin{align} \label{eq:nilsequ-1'}
  \sum_{\substack{d \leq N^{1/C_0}\\ \gcd(d,W)=1}}
& \frac{\mu(d)}{T}
  \sum_{\y \in \mathcal{Y}}
  \sum_{\tilde\x_d \in \mathcal{X}_{\y,d}}
  \sum_{\substack{\a \in \ZZ^{n-1}\\ |\a'|<\alpha}} 
  \sum_{j=1}^{k(\a)}\\
\nonumber
& \sum_{x \in \ZZ}
\Big\{ 
   \1_{ B^{-1/n}(Wq(d^2 \x + \tilde\x_d)+ \y) \in J_j^{(1)}(\a')} - 
   \1_{ B^{-1/n}(Wq(d^2 \x +  \tilde\x_d)+ \y) \in J_j^{(2)}(\a')} 
\Big\} 
  F\Big(g\big( P_{\a,\y} (x) \big)\Gamma\Big) \\
\nonumber  
+O\bigg(  
&  \sum_{\substack{d \leq N^{1/C_0}\\ \gcd(d,W)=1}}
  \frac{\mu(d)}{T''}
  \sum_{\y \in \mathcal{Y}}
  \sum_{\tilde\x_d \in \mathcal{X}_{\y,d}}
  \#\{\x \in \ZZ^n: B^{-1/n}(Wq(d^2 \x + \tilde\x_d)+ \y) \in 
      \partial_{\eps} \mathfrak{X}(1)\} \bigg).
\end{align}
Here, the error term accounts for all points in the
$B^{1/n}\eps$-neighbourhood of the boundary of $B^{1/n}
\mathfrak{X}(1)$,
that were excluded through the choice of intervals $I_j(\a)$.
Observe that we made use of the fact $\|F\|_{\infty} \leq 1$.
Since $\mathfrak{X}(1)$ is $(n-1)$-Lipschitz parametrisable, we have
$\vol(\partial_{\eps} \mathfrak{X}(1)) \asymp \eps$.
Together with applications of \eqref{eq:lifting} and
\eqref{eq:rho(d^2,0)-bound}, this shows that the error term is bounded by
\begin{align*}
\frac{1}{T} \sum_{d: \gcd(d,W)=1} 
\#\mathcal{Y} \rho(d^2,0) \frac{\eps B}{(Wqd^2)^n} 
&\ll \eps \frac{B}{TWq}\frac{\rho(Wq,A+Wb)}{(Wq)^{n-1}} 
    \sum_d \frac{\rho(d^2,0)}{d^{2n}} \\
&\ll \eps \frac{\rho(W,A)}{W^{n-1}}.
\end{align*}

Note that the set
$$
 \left\{ x \in \ZZ: 
    \frac{Wq(d^2 x + \tilde x_{d,n}) + y_n}{B^{1/n}} 
   \in J_j^{(1)}(\a') \right\}
$$
is a discrete interval of length
$$
 \# \left\{ x \in \ZZ: 
 \frac{Wq(d^2 x + \tilde x_{d,n}) + y_n}{B^{1/n}} 
 \in J_j^{(1)}(\a') \right\}
 \asymp \frac{B^{1/n}}{Wq d^2}.
$$

Thus, the total $\delta(N)^{E/C'''}$-equidistribution property of
$\Big(g\Big(P_{\a;\y;d}(x)\Big)
   \Gamma\Big)_{x \leq B^{1/n}/(Wqd^2)}$
implies that
\begin{align*}
\bigg|  \sum_{\substack{ x \in \ZZ: \\
   B^{-1/n}(Wq (d^2x + \tilde x_{d,n}) + y_n)  \\ \in J_j^{(1)}(\a')
  }}\hspace{-0.3cm}
  F\Big(g\big( P_{\a;\y;d} (x) \big)\Gamma\Big)\bigg|
&\ll 
 \delta(N)^{E/C'''}
 \frac{B^{1/n}}{Wq d^2}
 \|F\|_{\mathrm{Lip}}.
\end{align*}
The same holds for $J_j^{(1)}(\a')$ replaced by any non-empty 
$J_j^{(2)}(\a')$. 
Hence \eqref{eq:nilsequ-1'} is bounded by
\begin{align*}
&\ll T^{-1} \sum_{d \leq N^{1/C_0}}
  \#\mathcal Y 
  \frac{\rho(d^2,0)}{d^2}
  \Big(\frac{B^{1/n}}{Wq}\Big)^{n-1} \eps^{-1}
  B^{1/n}(Wq)^{-1}
 \delta(N)^{E/C'''} \|F\|_{\mathrm{Lip}} \\
& \qquad+(\varepsilon + \delta(N)^{O_{m_G,\ell, n}(E/C'')} + 
\delta(N)^{O_{m_G,\ell, n}(E/C''')})\frac{\rho(W,A)}{W^{n-1}} \\ 
&\ll T^{-1} \#\mathcal Y \frac{B}{(Wq)^{n}} \eps^{-1}
 \delta(N)^{E/C'''} \|F\|_{\mathrm{Lip}} \\
&\qquad+(\varepsilon + \delta(N)^{O_{m_G,\ell, n}(E/C'')} + 
\delta(N)^{O_{m_G,\ell, n}(E/C''')})\frac{\rho(W,A)}{W^{n-1}}\\ 
&\ll \frac{\rho(W,A)}{W^{n-1}} 
 \Big(
  \eps^{-1} \delta(N)^{E/C'''} \|F\|_{\mathrm{Lip}}
+\eps + \delta(N)^{O_{m_G,\ell, n}(E/C'')} + 
\delta(N)^{O_{m_G,\ell, n}(E/C''')} \Big),
\end{align*}
where we applied the bound \eqref{eq:rho(d^2,0)-bound}, 
the fact that $\#\mathcal{Y}= \rho(Wq,A+Wb)$, and the lifting property 
\eqref{eq:lifting}.
Choosing $\eps = \delta(N)^{E/2C'''}$ and $E_0=C'''$ completes the proof.
\end{proof}


\begin{thebibliography}{xx}
\bibitem{bm} T.D. Browning and L. Matthiesen,
Norm forms for arbitrary number fields as products of linear polynomials.
\texttt{arXiv:1307.7641}.

\bibitem{GT-longprimeaps} B.J. Green and T.C. Tao,
The primes contain arbitrarily long arithmetic progressions.
{\em Annals of Math.} {\bf 167} (2008),  481--547.

\bibitem{GT-linearprimes} B.J. Green and T.C. Tao,
Linear equations in primes.
{\em Annals of Math.} {\bf 171} (2010),  1753--1850.

\bibitem{GT-polyorbits} B.J. Green and T.C. Tao,
The quantitative behaviour of polynomial orbits on nilmanifolds.
{\em Annals of Math.} {\bf 175} (2012),  465--540.

\bibitem{GT-nilmobius} B.J. Green and T.C. Tao,
The M\"obius function is strongly orthogonal to nilsequences.
{\em Annals of Math.} {\bf 175} (2012),  541--566.

\bibitem{GT-erratum} B.J. Green and T.C. Tao,
On the quantitative behaviour of polynomial orbits on nilmanifolds 
-- Erratum.
\texttt{arXiv:1311.6170}.

\bibitem{GTZ} B.J. Green, T.C. Tao and T. Ziegler,
An inverse theorem for the Gowers $U_{s+1}[N]$-norm.
{\em Annals of Math.} {\bf 176} (2012),  1231--1372.

\bibitem{HW} Y. Harpaz and O. Wittenberg, 
On the fibration method for zero-cycles and rational points.
\texttt{arXiv:1409.0993}.

\bibitem{Leibman} A. Leibman, Polynomial mappings of groups (corrected version), 
{\em Israel J.\ Math.} {\bf 129} (2002), 29--60.

\bibitem{lmr} L. Matthiesen,
Linear correlations amongst numbers represented by positive definite 
binary quadratic forms.
{\em Acta Arith.} {\bf 154} (2012), 235--306.

\bibitem{lmm} L. Matthiesen,
Generalized Fourier coefficients of multiplicative functions.
\texttt{arXiv:1405.1018}.

\bibitem{lmf} L. Matthiesen,
A consequence of the factorisation theorem for polynomial nilsequences.
\texttt{arXiv:1509.xxxx}.

\end{thebibliography}
\end{document}